\documentclass[11pt,a4paper,reqno]{amsart}
\usepackage{amsfonts,amssymb,amsmath,amsthm,graphicx,color,amscd,xspace,verbatim}
\usepackage{scalerel}
\makeatletter

\usepackage{hyperref}

\makeatletter
\def\@tocline#1#2#3#4#5#6#7{\relax
  \ifnum #1>\c@tocdepth 
  \else
    \par \addpenalty\@secpenalty\addvspace{#2}%
    \begingroup \hyphenpenalty\@M
    \@ifempty{#4}{%
      \@tempdima\csname r@tocindent\number#1\endcsname\relax
    }{%
      \@tempdima#4\relax
    }%
    \parindent\z@ \leftskip#3\relax \advance\leftskip\@tempdima\relax
    \rightskip\@pnumwidth plus4em \parfillskip-\@pnumwidth
    #5\leavevmode\hskip-\@tempdima
      \ifcase #1
       \or\or \hskip 1em \or \hskip 2em \else \hskip 3em \fi%
      #6\nobreak\relax
      \dotfill
      \hbox to\@pnumwidth{\@tocpagenum{#7}}
    \par
    \nobreak
    \endgroup
  \fi}
\makeatother

\newtheorem{theorem}{Theorem}[section]
\newtheorem{lemma}[theorem]{Lemma}
\newtheorem{proposition}[theorem]{Proposition}

\newtheorem{corollary}[theorem]{Corollary}

\theoremstyle{definition}
\newtheorem{definition}[theorem]{Definition}

\theoremstyle{remark}
\newtheorem{remark}[theorem]{Remark}

\newcommand{\N}{{\mathbb N}}
\newcommand{\R}{{\mathbb R}}

\newcommand{\tr}{\mathrm{tr}^*}

\newcommand{\beqn}{\begin{eqnarray}}
\newcommand{\eeqn}{\end{eqnarray}}   
\newcommand{\beq}{\begin{eqnarray*}}
\newcommand{\eeq}{\end{eqnarray*}}

\newcommand{\be}{\begin{equation}}
\newcommand{\bel}[1]{\begin{equation}\label{#1}}
\newcommand{\ee}{\end{equation}}

\newcommand{\BA}{\begin{array}}
\newcommand{\EA}{\end{array}}
\newcommand{\BAN}{\renewcommand{\arraystretch}{1.2}
\setlength{\arraycolsep}{2pt}\begin{array}}
\newcommand{\BAV}[2]{\renewcommand{\arraystretch}{#1}
\setlength{\arraycolsep}{#2}\begin{array}}

\newcommand{\BSA}{\begin{subarray}}
\newcommand{\ESA}{\end{subarray}}
\newcommand{\BAL}{\begin{aligned}}
\newcommand{\EAL}{\end{aligned}}

\newcommand{\supp}{\mathrm{supp}\,}
\newcommand{\dist}{\mathrm{dist}\,}
\newcommand{\sign}{\mathrm{sign}}
\newcommand{\diam}{\mathrm{diam}\,}

\newcommand{\prt}{\partial}

\def\ga{\alpha}     \def\gb{\beta}

\def\gf{\phi}           
            
        \def\gn{\nu}

      \def\gw{\omega}
                \def\gz{\zeta}
     \def\Gd{\Delta}

\def\Gw{\Omega}              

\def\CD{{\mathcal D}}      
      
      \def\CL{{\mathcal L}}

   \def\BBE {\mathbb E}    
\def\BBG {\mathbb G}       
   \def\BBK {\mathbb K}    
       
\def\BBP {\mathbb P}   \def\BBR {\mathbb R}

\def\GTM {\mathfrak M}

\def\tr{\mathrm{tr}}

\def\div{\mathrm{div}}

\newcommand{\ei}{{\phi_{\xm }}}
\newcommand{\xa}{\alpha}
\newcommand{\xb}{\beta}
\newcommand{\xg}{\gamma}
\newcommand{\xG}{\Gamma}
\newcommand{\xd}{\delta}
\newcommand{\xD}{\Delta}
\newcommand{\xe}{\varepsilon}
\newcommand{\xz}{\zeta}

\newcommand{\xl}{\lambda}

\newcommand{\xm}{\mu}
\newcommand{\xn}{\nu}

\newcommand{\xr}{\rho}

\newcommand{\xf}{\phi}

\newcommand{\xo}{\omega}
\newcommand{\xO}{\Omega}

\newcommand{\myint}[2]{{\displaystyle \int_{#1}^{#2}}}

\newcommand{\ap}{{\xa_{\scaleto{+}{3pt}}}}
\newcommand{\am}{{\xa_{\scaleto{-}{3pt}}}}

\numberwithin{equation}{section}

\begin{document}

\title[Green kernel and Martin kernel of Schr\"odinger operators]{Green kernel and Martin kernel of Schr\"odinger operators with singular potential and application to the B.V.P. for linear elliptic equations}

\author{Konstantinos T. Gkikas}
\address{Konstantinos T. Gkikas, Department of Mathematics, National and Kapodistrian University of Athens, 15784 Athens, Greece}
\email{kugkikas@gmail.com}

\author[P.T. Nguyen]{Phuoc-Tai Nguyen}
\address{Department of Mathematics and Statistics, Masaryk University, Brno, Czech Republic}
\email{ptnguyen@math.muni.cz}


\begin{abstract} Let $\Omega \subset \mathbb{R}^N$ ($N \geq 3$) be a $C^2$ bounded domain and  $K \subset \Omega$ be a compact, $C^2$ submanifold in $\mathbb{R}^N$ without boundary, of dimension $k$ with $0\leq k < N-2$. We consider the Schr\"odinger operator $L_\mu = \Delta + \mu d_K^{-2}$ in $\Omega \setminus K$, where $d_K(x) = \dist(x,K)$. The optimal Hardy constant $H=(N-k-2)/2$ is deeply involved in the study of $-L_\mu$. When $\mu \leq H^2$, we establish sharp, two-sided estimates for Green kernel and Martin kernel of $-L_\mu$. We use these estimates to prove the existence, uniqueness and a priori estimates of the solution to the boundary value problem with measures for linear equations associated to $-L_\mu$.

	\medskip
	
	
	\medskip
	
	
\end{abstract}

\maketitle
\tableofcontents
\section{Introduction}
Let $\Omega \subset \R^N$ ($N \geq 3$) be a $C^2$ bounded domain and $K\subset\Omega$ be a compact, $C^2$ submanifold in $\R^N$ without boundary, of dimension $k$ with $0 \leq k < N-2$. Denote $d(x)=\dist(x,\partial \Omega)$ and $d_K(x) = \dist(x,K)$. In this paper, we study the Schr\"odinger operator
\begin{equation} \label{Lmu} L_\mu = L_\mu^{\Omega,K}:=\Delta + \frac{\mu}{d_K^2}
\end{equation}
in $\Omega \setminus K$, where $\mu \in \BBR$ is a parameter. Here $\mu/ d_K^2$ is a singular potential.

Our first purpose is to deal with the question of two-sided estimates on the Green kernel associated to $-L_\mu$ in $\Omega \setminus K$.

It is well known that in the free potential case in $\xO$, i.e. $\mu=0$ and $L_0=\Delta$ in $\xO$, this question has been completely treated and sharp estimates, up to boundary, have been obtained in (see, e.g., \cite{Zha}). See also \cite{BVi, MVbook} for relevant estimates.

The special case $\mu \neq 0$, $k=0$ and $K=\{0\} \subset \Omega$ has attracted a lot of attention since in this case $L_\mu$ is a singular operator involving Hardy-Leray potential $\mu |x|^{-2}$. Global upper estimate and local lower estimate (i.e. in compact subsets of $\Omega$) on the Green kernel associated to $-L_\mu$ was derived by Chen, Quaas and Zhou in \cite[Lemma 4.1 and Remark 4.1]{Chen2} due to the two-sided estimate on the corresponding heat kernel established by Filippas, Moschini and Tertikas in \cite{FMoT2}.

The general case is more challenging and  less understood. In this case, the analysis is more intricate and relies strongly on the geometrical properties of the set $K$ and $\xO$, which is closely linked to the optimal Hardy constant
$${\mathcal C}_{\xO,K}:=\inf_{\varphi \in H^1_0(\xO)}\frac{\int_\xO|\nabla \varphi|^2dx}{\int_\xO d_K^{-2}\varphi^2dx}.$$

It is well known that ${\mathcal C}_{\xO,K}\in (0,H^2]$ (see e.g. D\'avila and Dupaigne \cite{DD1, DD2} and Barbatis, Filippas and Tertikas \cite{BFT}) where
\begin{equation} \label{H}
H:=\frac{N-k-2}{2}.
\end{equation}

When $K=\{0\} \subset \Omega$, it is classical that ${\mathcal C}_{\Omega,\{0\}}=\left(\frac{N-2}{2} \right)^2$. In general, ${\mathcal C}_{\Omega,K}=H^2$ provided that $-\Delta d_K^{2+k-N} \geq 0$ in the sense of distributions in $\Omega \setminus K$ or $\Omega=K_\beta$ with $\beta$  small enough (see \cite{BFT}), where
$$K_\beta :=\{ x \in \R^N \setminus K: d_K(x) < \beta \}.$$

For $\mu \leq H^2$, let $\am$ and $\ap$ are the roots of the algebraic equation $\ga^2 - 2H\ga + \mu=0$, i.e.
\bel{apm}
\am:=H-\sqrt{H^2-\mu}, \quad \ap:=H+\sqrt{H^2-\mu}.
\ee
Note that $\am\leq H\leq\ap<2H$ and $\am \geq 0$ if and only if $\mu \geq 0$.

For any $\mu \leq H^2$, it follows from  \cite[Lemma 2.4 and Theorem 2.6]{DD1} and \cite[page 337, Lemma 7, Theorem 5]{DD2} that
\be\label{Lin01i} \lambda_\mu:=\inf\left\{\int_{\Gw}\left(|\nabla u|^2-\frac{\xm }{d_K^2}u^2\right)dx: u \in C_c^1(\Omega), \int_{\Gw} u^2 dx=1\right\}>-\infty
\ee
and the corresponding eigenfunction $\phi_\mu$, with normalization $\| \phi_\mu \|_{L^2(\Omega)}=1$, satisfies two-sided estimate $\phi_\mu \approx d\,d_K^{-\am}$ in $\Omega \setminus K$ (see subsection \ref{subsect:eigen} for more detail). A combination of this estimate for the eigenfunction and results of Filippas, Moschini and Tertikas in \cite[Proposition 2.8 and Theorem 1.3]{FMoT}, leads to the existence, as well as two-sided estimate, of a heat kernel $h(t,x,y)$ associated to $\partial_t-L_\mu$ when $\lambda_\mu>0$. Set
\begin{equation} \label{Gheat} G_{\mu}(x,y)=\int_0^\infty h(t,x,y)dt, \quad x,y \in \Omega \setminus K,\; x \neq y.
\end{equation}

\begin{theorem}[Green kernel]  \label{Greenkernel} Assume $0\leq k<N-2$, $\mu \leq H^2$ and $\lambda_\mu>0$.

{\sc I. Existence of the minimal Green kernel.} $G_\mu$ is the minimal  Green kernel  of $-L_\mu$ in $\Omega \setminus K$, i.e. for any $y \in \Omega \setminus K$, $G_\mu(\cdot,y)$ is the minimal solution of
	\bel{greeneq}
	-L_\mu u =\delta_y \quad \text{in } \Omega \setminus K,
	\ee
in the sense of distributions, where $\delta_y$ denotes the Dirac measure concentrated at $y$. \smallskip	

{\sc II. Two-sided estimates.}
	
	(i) If $\mu< \left( \frac{N-2}{2}\right)^2$ then for any $x,y \in \Omega \setminus K$, $x \neq y$,
	\bel{Greenesta}
	G_{\mu}(x,y)\approx |x-y|^{2-N} \left(1 \wedge \frac{d(x)d(y)}{|x-y|^2}\right) \left(1 \wedge \frac{d_K(x)d_K(y)}{|x-y|^2} \right)^{-\am}.
	\ee
	
	(ii) If $k=0$, $K=\{0\}$ and $\mu = \left( \frac{N-2}{2}\right)^2$ then for any $x,y \in \Omega \setminus K$, $x \neq y$,
	\bel{Greenestb} \BAL
	G_{\mu}(x,y) &\approx |x-y|^{2-N} \left(1 \wedge \frac{d(x)d(y)}{|x-y|^2}\right) \left(1 \wedge \frac{|x||y|}{|x-y|^2} \right)^{-\frac{N-2}{2}}\\
	& \quad +(|x||y|)^{-\frac{N-2}{2}}\left|\ln\left(1 \wedge \frac{|x-y|^2}{d(x)d(y)}\right)\right|.
	\EAL \ee
	
	Here the notation "$\approx$" is introduced in the list of notations at the end of this section. The implicit constants in \eqref{Greenesta} and \eqref{Greenestb} depend on $N,\Omega,K,\mu$.
\end{theorem}

\begin{remark}
(i) We note that, uniqueness may not hold true for \eqref{greeneq}. Therefore, in the sequel, \textit{by Green kernel we mean the minimal Green kernel $G_\mu$, which is defined in \eqref{Gheat}}.

(ii) Note that, in Theorem \ref{Greenkernel} (i), the critical case $k>0$ and $\mu=H^2$ is included.

(iii) One of the main assumptions in this paper is $\lambda_\mu>0$, which is fulfilled for instance if $\mu<{\mathcal C}_{\Omega,K}$. In the critical case $\mu=H^2$, $\lambda_{H^2}>0$ if $-\Delta d_K^{2+k-N} \geq 0$ in the sense of distributions in $\Omega \setminus K$. Finally, if $\Omega=K_\beta$ then $-\Delta d_K^{2+k-N} \geq 0$ provided $\beta$ is small enough. See \cite{BFT} for the proof of these results as well as for other domains satisfying $-\Delta d_K^{2+k-N} \geq 0.$	

(iii) Estimates \eqref{Greenesta} and \eqref{Greenestb} cover the ones  in \cite{Chen1, Chen2} for the case $k=0$, $K=\{0\}$ and are sharper than the estimate in  \cite[Corollary 7.3]{DD1}.

\end{remark}

Let $\beta_0$ be the constant in \eqref{straigh}. Let $\eta_{\beta_0}$ be a smooth function such that $0 \leq \eta_{\beta_0} \leq 1$, $\eta_{\beta_0}=1$ in $\overline{K}_{\frac{\xb_0}{4}}$ and $\supp \eta_{\beta_0} \subset K_{\frac{\beta_0}{2}}$. We define
\bel{W} W(x):=\left\{ \BAL &d_K(x)^{-\ap},\qquad&&\text{if}\;\mu <H^2, \\
	&d_K(x)^{-H}|\ln d_K(x)|,\qquad&&\text{if}\;\mu =H^2,
	\EAL \right. \quad x \in \Omega \setminus K,
\ee
	and
\bel{tildeW}
	\tilde W:=1-\eta_{\beta_0}+\eta_{\beta_0}W \quad \text{in } \Omega \setminus K.
\ee

Let $\BBG_\mu$ be the Green operator, i.e.
$$ \BBG_\mu[\tau](x) = \int_{\Omega \setminus K} G_\mu(x,y)\tau(y), \quad \tau \in \GTM(\Omega \setminus K).
$$

As it can be seen in Lemma \ref{existence1b}, the Green operator is a crucial tool in solving nonhomogeneous linear equation associated to $-L_\mu$ with ``zero datum'' on $\partial (\Omega \setminus K) = \partial \Omega \cup K$.
More precisely, for $f \in L^\infty(\Omega)$, $\BBG_\mu[f]$ solves equation $-L_\mu u = f$ in $\Omega \setminus K$ with zero boundary condition in the sense
$$
	\lim_{\dist(x,F)\to 0}\frac{\BBG_\mu[f](x)}{\tilde W(x)}=0,\quad\forall\; \text{compact} \; F\subset \partial \Omega \cup K.
$$

To study linear equations with more general boundary data on $\partial(\Omega \setminus K) = \partial \Omega \cup K$, we use the \textit{$L_\mu$-harmonic measures} which is given below. Let $z \in \Omega \setminus K$ and $h\in C(\partial\Omega \cup K)$ and denote $\CL_{\mu ,z}(h):=v_h(z)$ where $v_h$ is the unique solution of the Dirichlet problem (see Lemma \ref{mainlemma1})
\be \label{linear} \left\{ \BAL
L_{\mu}v&=0\qquad \text{in}\;\;\xO\setminus K\\
v&=h\qquad \text{on}\;\;\partial\xO\cup K.
\EAL \right. \ee
Here the boundary value condition in \eqref{linear} is understood in the sense that

$$
	\lim_{x\in\Omega \setminus K,\;x\rightarrow y\in\partial \Omega \cup  K}\frac{v(x)}{\tilde W(x)}=h(y)\qquad\text{uniformly w.r.t. } y\in\partial\Omega \cup K.
$$
The mapping $h\mapsto \CL_{\mu,z}(h)$ is a linear positive functional on $C(\partial\Omega \cup K)$ (see Lemma \ref{comparison}). Thus there exists a unique Borel measure on $\partial\Omega \cup K$, called {\it $L_{\mu}$-harmonic measure in $\partial \Omega \cup K$ relative to $z$} and  denoted by $\omega^{z}$, such that
$$v_{h}(z)=\int_{\partial\Omega\cup K}h(y) d\omega^{z}(y).$$

Let $x_0 \in \Omega \setminus K$ be a fixed reference point. By Harnack inequality, the measures $\omega^{x}$ and $\omega^{x_0}$ are mutually absolutely continuous for any $x \in\Omega \setminus K$, hence we can define the Radon-Nikodym derivative
\bel{RN-derivative}
K_\mu(x,\xi):=\frac{d\omega^{x}}{d\omega^{x_0}}(\xi),\quad \omega^{x_0}-\text{a.e. }\;\xi\in\partial\Omega \cup K.
\ee

Let us now give a definition of kernel functions of $-L_\mu$ at $\xi$ which plays an important role in the sequel.
\begin{definition} \label{kernel}
	A function $\mathcal{K}$ defined in $(\Omega \setminus K) \times (\partial \Omega \cup K)$ is called a kernel function of $-L_\mu$  with pole at $\xi \in\partial\Omega \cup K$ and with basis at $x_0\in\Omega \setminus K$ if\smallskip
	
	(i) $\mathcal{K}(\cdot,\xi)$ is $L_{\mu }$-harmonic in $\Omega \setminus K$,\smallskip
	
	(ii) $\frac{\mathcal{K}(\cdot,\xi)}{\tilde W(\cdot)}$ can be extended as a continuous function on $\overline{\xO}\setminus\{\xi\}$ and for any $P \in(\prt\Gw\cup K)\setminus\{\xi\}$,
	$$\lim_{x \in \Omega \setminus K,\; x \to P}\frac{\mathcal{K}(x,\xi)}{\tilde W(x)}=0,$$\smallskip
	
	(iii) $\mathcal{K}(x,\xi)>0$ for every $x\in\xO\setminus K$ and $\mathcal{K}(x_0,\xi)=1$.
\end{definition}

Using main properties of the $L_{\mu }$-harmonic measures, which are established in Section 6, we prove  the existence and uniqueness of the kernel function $\mathcal{K}$ (see Proposition \ref{uniq}). Furthermore, we will show that (see Proposition \ref{martincorresp}) the following convergence holds
\be\label{martindef}
K_{\mu}(x,\xi)=\lim_{\Omega \setminus K \ni y\to\xi}\frac{G_\mu(x,y)}{G_\mu(x_0,y)}\qquad\forall\xi\in\partial \Omega \cup K,
\ee
which means that $K_\mu$ is the Martin kernel of $-L_\mu$ in $\Omega \setminus K$.

Let us present the main properties of the Martin kernel.

\begin{theorem}[Martin kernel] \label{Martin} Assume $0\leq k<N-2$, $\mu \leq H^2$ and $\lambda_\mu>0$.

{\sc I. Continuity.} For any $x \in \Omega \setminus K$, the function $\xi \mapsto K_\mu(x,\xi)$ is continuous on $\partial \Omega \cup K$.

{\sc II. Two-sided estimates.}

(i) If $\mu< \left( \frac{N-2}{2}\right)^2$ then
\be \label{Martinest1}
K_{\mu}(x,\xi) \approx\left\{
\BAL
&\frac{d(x)d_K(x)^{-\am}}{|x-\xi|^N},\quad &&\text{if } x \in \Omega \setminus K,\;  \xi \in \partial\xO \\
&\frac{d(x)d_K(x)^{-\am}}{|x-\xi|^{N-2-2\am}}, &&\text{if } x \in \Omega \setminus K,\; \xi \in K.
\EAL \right.
\ee

(ii) If  $k=0$, $K=\{0\}$ and $\mu= \left( \frac{N-2}{2}\right)^2$ then
\be\label{Martinest2}
K_{\mu}(x,\xi) \approx\left\{
\BAL
&\frac{d(x)|x|^{-\frac{N-2}{2}}}{|x-\xi|^N},\quad &&\text{if } x \in \Omega \setminus K,\; \xi \in \partial\Omega, \\
&d(x)|x|^{-\frac{N-2}{2}}\left|\ln\frac{|x|}{\CD_\Omega}\right|, &&\text{if } x \in \Omega \setminus K,\; \xi=0,
\EAL \right.
\ee
where $\CD_{\Omega}:=2\sup_{x \in \Omega}|x|$.

\end{theorem}

Estimates \eqref{Martinest1}, \eqref{Martinest1} are novel and show distinct behaviours of Martin kernel, according to whether $\xi \in \partial \Omega$ or $\xi \in K$.

It is interesting to note that two-sided estimates of the Green kernel and Martin kernel were also studied for Schr\"odinger operators of the form $-L_{\mu V}=-\Delta -  \mu V$ where the potential $V$ may blowup on the boundary $\partial \Omega$. When $V(x)=d(x)^{-2}$, the existence, as well as sharp estimates, of the Green kernel was obtained by Filippas, Moschini and Tertikas \cite{FMoT2} via the study of the respective parabolic problem, while the existence of the Martin kernel of $-L_\mu$ was established by Ancona \cite{An} in the subcritical case, i.e $\mu< \frac{1}{4}$, and by Gkikas-V\'{e}ron \cite{GkV} in the critical case $\mu=\frac{1}{4}$. In \cite{Mar1}, Marcus  dealt with a more general potential $V$ satisfying $|V(x)| \leq c\, d(x)^{-2}$ and obtained two-sided estimates on the Green kernel and Martin kernel  in terms of the first eigenfunction. In the case $V(x)=d_F(x)^{-2}$, where $F \subset \partial \Omega$ is a submanifold of dimension $0 \leq k \leq N-2$, these bounds were then exploited by Marcus and Nguyen \cite{Mar-Ng} to derive estimates of Green and Martin kernel on layers near the boundary $\partial \Omega$, which are in turn used to study respective linear and semilinear elliptic equations. Very recently, Marcus \cite{Mar2} has established two-sided estimates for positive $L_{\mu V}$-subharmonic and $L_{\mu V}$-superharmonic functions with $V$ satisfying $|V(x)| \leq cd(x)^{-2}$ and  provided a theory of linear equations associated to $L_{\mu V}$ which cover several results in \cite{MarNgu,Mar-Ng}. The case $0 \in \partial \Omega$ and $V(x)=|x|^{-2}$ was treated by Chen and V\'eron in \cite{CheVer2} where they constructed a Poisson kernel vanishing at $0$ and a singular kernel with a singularity at $0$. Relevant works on semilinear elliptic equations involving $-L_{\mu V}$ can be found in \cite{GkV,MarNgu,MarMor,Mar-Ng,MVbook,CheVer1, CheVer3}.

Next we provide the Representation theorem which states that there is a (1-1) correspondence between the class of positive $L_\mu$-harmonic functions in $\Omega \setminus K$ and the set $\GTM^+(\partial \Omega \cup K)$ of positive bounded measures on $\partial (\Omega \setminus K) = \partial \Omega \cup K$.

\begin{theorem}[Representation Theorem] \label{th:Rep} For any $\nu \in \GTM^+(\partial \Omega \cup K)$, the function $\BBK_\mu[\nu]$ is a positive $L_\mu$-harmonic function (i.e. $L_\mu \BBK_\mu[\nu]=0$) in the sense of distributions in $\Omega \setminus K$. Conversely, for any positive $L_\mu$-harmonic function $u$ (i.e. $L_\mu u = 0$) in the sense of distributions in $\Omega \setminus K$, there exists a unique measure $\nu \in \GTM^+(\partial \Omega \cup K)$ such that $u=\BBK_\mu[\nu]$.
\end{theorem}

In general, in order to characterize the behavior of a function on $\partial \Omega \cup K$, we introduce a notion of boundary trace which is defined in a \textit{dynamic way}.

\begin{definition}[Boundary trace] \label{def:trace}
	A function $u$ possesses a \emph{boundary trace}  if there exists a measure $\gn\in\GTM(\prt\Gw\cup K)$ such that for any smooth exhaustion  $\{ O_n \}$ of $\xO\setminus K$, there  holds
	\be\label{trab}
	\lim_{n\rightarrow\infty}\int_{ \partial O_n}\gf u\, d\xo_{O_n}^{x_0}=\int_{\partial \xO\cup K} \gf \,d\xn \quad\forall \gf \in C(\overline{\xO}).
	\ee
	The boundary trace of $u$ is denoted by $\tr(u)$.
\end{definition}

This notion allows to describe the boundary behavior of the Green kernel and Martin kernel, which can be seen  in the following proposition.

\begin{proposition}[Boundary trace of Green kernel and Martin kernel] \label{traceKG} ~~

(i) For any $\nu \in \GTM(\partial \Omega \cup K)$, $\tr(\BBK_\mu[\nu])=\nu$.

(ii) For any $\tau \in \GTM(\Omega \setminus K;\ei)$, $\tr(\BBG_\mu[\tau])=0$.
\end{proposition}

Theorem \ref{th:Rep} and Proposition \ref{traceKG} are key ingredients in obtaining the following properties of $L_\mu$-subharmonic and $L_\mu$-superharmonic functions.

\begin{theorem}\label{subsuperhar}
	(i) Let $u$ be a positive $L_\mu$-superharmonic function in the sense of distributions in $\Omega \setminus K$. Then $u \in L^1(\Omega;\ei)$ and there exist $\tau \in \GTM^+(\Omega \setminus K;\ei)$ and $\nu \in \GTM^+(\partial\Omega \cup K)$ such that
	\bel{reprweaksol} u=\BBG_{\mu}[\tau]+\BBK_{\mu}[\nu].
	\ee
	In particular, $u \geq \BBK_\mu[\nu]$ in $\Omega \setminus K$ and $\tr(u)=\nu$.
	
	(ii) Let $u$   be a positive $L_\mu$-subharmonic function  in the sense of distributions in $\Omega \setminus K$. Assume that there exists a positive $L_\mu$-superharmonic function $w$ such that $u \leq w$ in $\Omega \setminus K$. Then $u \in L^1(\Omega;\ei)$ and there exist $\tau \in \GTM^+(\Omega \setminus K;\ei)$ and $\nu \in \GTM^+(\partial\Omega \cup K)$ such that
	\bel{vGK2} u+\BBG_{\mu}[\tau]=\BBK_{\mu}[\nu].
	\ee
	In particular, $u \leq \BBK_\mu[\nu]$ in $\Omega \setminus K$ and $\tr(u)=\nu$.
\end{theorem}

We are ready to study the boundary value problem for linear equations.
\begin{definition} \label{weaksol-LP}
 Let $\tau\in\mathfrak{M}(\xO\setminus K;\ei)$ and $\nu \in \mathfrak{M}(\partial\xO\cup K)$. We say that $u$  is a weak solution of
\begin{equation}\label{NHLP} \left\{ \BAL
- L_\mu u&=\tau\qquad \text{in }\;\Omega \setminus K,\\
\tr(u)&=\xn,
\EAL \right. \end{equation}
if $u\in L^1(\Omega;\ei)$ and it satisfies
\be \label{lweakform}
	- \int_{\Omega} u L_{\xm }\zeta \, dx=\int_{\Omega \setminus K} \zeta \, d\tau - \int_{\Omega} \mathbb{K}_{\xm}[\xn]L_{\xm }\zeta \, dx
	\qquad\forall \zeta \in\mathbf{X}_\xm(\xO\setminus K),
	\ee
	where the space of test function ${\bf X}_\mu(\Gw\setminus K)$ is defined by
	\bel{Xmu} {\bf X}_\mu(\Gw\setminus K):=\{ \zeta \in H_{loc}^1(\Omega \setminus K): \phi_\mu^{-1} \zeta \in H^1(\Gw;\phi_\mu^{2}), \, \phi_\mu^{-1}L_\mu \zeta \in L^\infty(\Omega)  \}.
	\ee
\end{definition}

\begin{theorem} \label{Existence-LP} Assume $0\leq k<N-2$, $\mu \leq H^2$ and $\lambda_\mu>0$.

{\sc I. Existence and uniqueness.} For any $\tau \in \GTM(\Omega \setminus K;\ei)$ and $\nu \in \GTM(\partial \Omega \cup K)$, there exists a unique weak solution $u$ of \eqref{NHLP}.
The solution $u$ can be decomposed as in \eqref{reprweaksol}. In particular, $\BBG_\mu[\tau]$ is the weak unique solution of \eqref{NHLP} with $\nu=0$ and $\BBK_\mu[\nu]$ is the unique weak solution of \eqref{NHLP} with $\tau=0$.

{\sc II. A priori estimates.} There exists a positive constant $C=C(N,\Omega,K,\mu)$ such that
\begin{equation} \label{esti222}
\| u \|_{L^1(\Omega; \ei)} \leq \frac{1}{\lambda_\mu} \| \tau \|_{\GTM(\Omega \setminus K; \ei)} + C\| \nu \|_{\GTM(\partial \Omega \cup K)}.
\end{equation}
In addition, for any $\zeta \in {\bf X}_\mu(\Omega \setminus K)$ and $\zeta \geq 0$, the following estimates are valid
\be\label{est-L-1}
-\int_{\Omega}|u|L_\mu \zeta \, dx\leq \int_{\Omega \setminus K}\zeta d|\tau|-
\int_{\Omega}\BBK_\mu[|\nu|] L_\mu \zeta \, dx,
\ee
\be\label{est-L-2}
-\int_{\Omega}u_+ L_\mu \zeta \, dx\leq \int_{\Omega \setminus K}\zeta d\tau_+-
\int_{\Omega}\BBK_\mu[\nu_+] L_\mu \zeta \, dx.
\ee
\end{theorem}

Estimates on the Green and Martin kernels and the theory for linear equations associated to $-L_\mu$ are crucial tools in the study of respective semilinear elliptic equations which will be presented in a forthcoming paper.
\medskip

\noindent \textbf{Organization of the paper.} In Section 2, we introduce main assumptions on $K$ and present the background of the eigen pair of $-L_\mu$. In section 3, we construct local sub and super $L_\mu$-harmonic functions and prove Harnack type inequality. Section 4 is devoted to the proof of Theorem \ref{Greenkernel}. In section 5, we establish the solvability and a priori estimate for linear equations with continuous boundary data. In section 6, we demonstrate Theorems \ref{Martin} and \ref{th:Rep}.  Finally, in section 7, we prove Proposition \ref{traceKG} and Theorems \ref{subsuperhar} and \ref{Existence-LP}. \medskip

\noindent \textbf{Notations.}

\begin{itemize}
\item The notation $A \gtrsim B$ (resp. $A \lesssim B$) means $A \geq c\,B$ (resp. $A \leq c\,B$) where the implicit $c$ is a positive constant depending on some initial parameters. If $A \gtrsim B$ and $A \lesssim B$, we write $A \approx B$. \textit{Throughout the paper, most of the implicit constants depend on some (or all) of the initial parameters such as $N,\Omega,K,k,\mu$ and we will omit these dependences in the notations (except when it is necessary).}

\item Let $\phi$ be a positive continuous function in $\Omega \setminus K$.
Denote by $\GTM(\Omega \setminus K;\phi)$ the space of Radon measures $\tau$ in $\Omega \setminus K$ such that $\int_{\Omega \setminus K}\phi d|\tau|<\infty$ and by $\GTM^+(\Omega \setminus K;\phi)$ its positive cone. We also denote by  $\GTM(\partial \Omega \cup K)$ the space of Radon measures on $\partial \Omega \cup K$ and by $\GTM^+(\partial \Omega \cup K)$ its positive cone.
\item For $a,b \in \BBR$, denote $a \wedge b = \min\{a,b\}$, $a \lor b =\max\{a,b \}$.
\item For $\beta>0$, $ \Omega_{\beta}=\{ x \in \Omega: d(x) < \beta\}$, $K_{\beta}=\{ x \in \R^N \setminus K:  d_K(x)<\beta \}$.
\item We denote by $c,c_1,C...$ the constants which depend on initial parameters and may change from one appearance to another.
\end{itemize} \medskip

\noindent \textbf{Acknowledgements.} P.-T. Nguyen is supported by Czech Science Foundation, project GJ19 -- 14413Y. Part of this research was carried out by P.-T. Nguyen during a visit at the Hausdorff Research Institute for Mathematics (HIM), through the Trimester Program ``Evolution of Interfaces". P.-T. Nguyen gratefully acknowledges the support of the HIM. The authors wish to thank Professor L. V\'eron for many useful comments which help to improve the manuscript.

\section{Preliminaries}

\subsection{Assumptions on $K$.} Throughout this paper, we assume that $K \subset \Omega$ is a $C^2$ compact submanifold in $\mathbb{R}^N$ without boundary, of dimension $k$, $0\leq k < N-2$. When $k = 0$ we assume that $K = \{0\} \subset \Omega$.

For $x=(x_1,...,x_k,x_{k+1},...,x_N) \in \R^N$, we write $x=(x',x'')$ where $x'=(x_1,..,x_k) \in \R^k$ and $x''=(x_{k+1},...,x_N) \in \R^{N-k}$. For $\beta>0$, we denote by $B_{\beta}^k(x')$ the ball  in $\R^k$ with center at $x'$ and radius $\beta.$ For any $\xi\in K$, we set
\begin{align} \nonumber  K_\beta &:=\{ x \in \R^N \setminus K: d_K(x) < \beta \}, \\
\label{Vxi}
V(\xi,\xb)&:=\{x=(x',x''): |x'-\xi'|<\beta,\; |x_i-\Gamma_i^\xi(x')|<\xb,\;\forall i=k+1,...,N\},
\end{align}
for some functions $\Gamma_i^\xi: \R^k \to \R$, $i=k+1,...,N$.

Since $K$ is a $C^2$ compact submanifold in $\mathbb{R}^N$ without boundary, we may assume the existence of $\xb_0$ such that the followings hold.

\begin{itemize}
\item $K_{6\beta_0}\Subset \Omega$ and for any $x\in K_{6\beta_0}$, there is a unique $\xi \in K$  satisfies $|x-\xi|=d_K(x)$.

\item $d_K \in C^2(K_{4\beta_0})$, $|\nabla d_K|=1$ in $K_{4\beta_0}$ and there exists $g\in L^\infty(K_{4\beta_0})$ such that
\be \label{laplaciand}
\Delta d_K(x)=\frac{N-k-1}{d_K(x)}+g(x) \quad \text{in } K_{4\beta_0} .
\ee
(See \cite[Lemma 2.2]{Vbook} and \cite[Lemma 6.2]{DN}.)

\item For any $\xi \in K$, there exist $C^2$ functions $\Gamma_i^\xi \in C^2(\R^k;\R)$, $i=k+1,...,N$, such that (upon relabeling and reorienting the coordinate axes if necessary), for any $\beta \in (0,6\beta_0)$, $V(\xi,\beta) \subset \Omega$ and
\begin{equation} \label{straigh}
V(\xi,\beta) \cap K=\{x=(x',x''): |x'-\xi'|<\beta,\;  x_i=\Gamma_i^\xi (x'), \; \forall i=k+1,...,N\}.
\end{equation}

\item There exist $\xi^{j}$, $j=1,...,m_0$, ($1 \leq m_0 \in \N$) and $\beta_1 \in (0, \beta_0)$ such that
\be \label{cover}
K_{2\xb_1}\subset \cup_{i=1}^{m_0} V(\xi^i,\beta_0)\Subset \Omega.
\ee
\end{itemize}

Now set
\bel{dist2} \xd_K^\xi(x):=\left(\sum_{i=k+1}^N|x_i-\Gamma_i^\xi(x')|^2\right)^{\frac{1}{2}}, \qquad x=(x',x'')\in V(\xi,4\beta_0).\ee

Then we see that there exists a constant $C=C(N,K)$ such that
\be\label{propdist}
d_K(x)\leq	\xd_K^{\xi}(x)\leq C \| K \|_{C^2} d_K(x),\quad \forall x\in V(\xi,2\beta_0),
\ee
where $\xi^j=((\xi^j)', (\xi^j)'') \in K$, $j=1,...,m_0$, are the points in \eqref{cover} and
\begin{equation} \label{supGamma}
\| K \|_{C^2}:=\sup\{  || \Gamma_i^{\xi^j} ||_{C^2(B_{5\beta_0}^k((\xi^j)'))}: \; i=k+1,...,N, \;j=1,...,m_0 \} < \infty.
\end{equation}
Moreover, $\beta_1$ can be chosen small enough such that for any $x \in K_{\beta_1}$,
\bel{BinV} B(x,\beta_1) \subset V(\xi,\beta_0),
\ee
where $\xi \in K$ satisfies $|x-\xi|=d_K(x)$.

\subsection{Eigenvalue of $-L_\mu$} \label{subsect:eigen}
We recall that $0 \leq k <N-2$ and
$$H=\frac{N-k-2}{2}, \quad \am:=H-\sqrt{H^2-\mu}, \quad \ap:=H+\sqrt{H^2-\mu}.
$$

We summarize below main properties of the first eigenfunction of the operator $-L_\mu$ in $\Omega \setminus K$ from \cite[Lemma 2.4 and Theorem 2.6]{DD1} and \cite[page 337, Lemma 7 and Theorem 5]{DD2}.

(i) For any $\mu \leq H^2$, it is known that
\be\label{Lin01} \lambda_\mu:=\inf\left\{\int_{\Gw}\left(|\nabla u|^2-\frac{\xm }{d_K^2}u^2\right)dx: u \in C_c^1(\Omega), \int_{\Gw} u^2 dx=1\right\}>-\infty.
\ee

\smallskip

(ii) If $\mu < H^2$, there exists a minimizer $\gf_{\xm }$ of \eqref{Lin01} belonging to $H^1_0(\Gw)$. Moreover, it satisfies $-L_\mu \phi_\mu= \lambda_\mu \phi_\mu$  in $\Omega \setminus K$ and
\be\label{Lin1}
\gf_{\xm }\approx d_K^{-\am} \quad  \text{in } K_{\xb_0}.
\ee

\smallskip

(iii) If $\xm =H^2$, there is no minimizer of \eqref{Lin01} in $H_0^1(\Gw)$, but there exists a nonnegative function $\phi_{H^2}\in H_{loc}^1(\xO)$  such that $-L_{H^2}\phi_{H^2}=\lambda_{H^2}\phi_{H^2}$ in the sense of distributions in $\Omega \setminus K$ and
\be\label{Lin2}\phi_{H^2}\approx d_K^{-H}  \quad  \text{in } K_{\xb_0}. \ee
In addition, the function $\psi_{H^2}=d_K^{-H}\xf_{H^2}$ belongs to $H^1_0(\Gw; d_K^{-2H}).$

From \eqref{Lin1} and \eqref{Lin2} we deduce that
\be \label{eigenfunctionestimates}
\xf_\xm \approx d\,d^{-\am}_K \quad \text{in } \Omega \setminus K.
\ee

\subsection{Weighted Sobolev spaces on $\Omega \setminus K$}
It is known that if $p < N-k$ then $W^{1,p}_0(\xO\setminus K)=W_0^{1,p}(\xO)$ (see e.g. \cite{BFT}). Next we give some properties of the spaces of test functions.
\begin{proposition} \label{test} (i) $H_0^1(\xO\setminus K; d^2d_K^{-2\am})=H^1(\xO; d^2d_K^{-2\am})$. In particular, $C_0^\infty(\Omega \setminus K)$ is dense in $H^1(\Omega;d^2d_K^{-2\am})$.

(ii) $H_0^1(\xO\setminus K; d_K^{-2\am})=\{u\in H^1(\xO; d_K^{-2\am}):u|_{\partial\xO}=0\},$
	where $u |_{\partial\xO}$ denotes the Sobolev trace of $u$ on $\partial\xO$.
\end{proposition}
\begin{proof} (i) By applying \cite[Theorem 3.6]{FMoT} with $d_1=d,\;\xa_1=1,d_{n-k}=d_K,\xa_{n-k}=-\am$ and $\xa_i=0$ for all $i=\overline{1,n}$, $i\neq1,n-k,$ we can prove (i).

(ii) In view of the proof of  \cite[Theorem 3.6]{FMoT} and (i), we obtain the desired result.
\end{proof}

\section{$L_\mu$-sub and super harmonic functions}
In the sequel, we assume that $0 \leq k < N-2$ and $\mu \leq H^2$. In this subsection, we construct $L_\xm$-subharmonic and $L_\mu$-superharmonic functions which are defined in a ``small" neighborhood of $K.$ Let us give first the definition of $L_\xm$-harmonic function.
\begin{definition}
	Let $G\subset\xO\setminus K$ be open. We say that a function $u$ is $L_\mu$-subharmonic in $G$ if $u \in  H^1_{loc}(G)\cap C(G)$ and
	\bel{def:subhar}
	\int_G\nabla u \cdot \nabla\psi dx-\xm \int_G\frac{u \psi}{d^2_K} dx\leq0\qquad\forall\psi\in H_c^1(G), \;\psi\geq 0
	\ee
	where $H_c^{1}(G)$ denotes the subspace of $H^1(G)$ of functions with compact support in $G$.
	
	Similarly, a function $u$ is  $L_\mu$-superharmonic  in $G$ if $u \in  H^1_{loc}(G)\cap C(G)$ and $u$ satisfies \eqref{def:subhar} with
	with $``\leq"$ replaced by $``\geq"$.
	
	A function $u$ is $L_\mu$-harmonic in $G$ if $u$ is $L_\mu$-subharmonic and $L_\mu$-superharmonic in $G$.
\end{definition}

\begin{lemma}[Local $L_\mu$-subharmonic functions and $L_\mu$-superharmonic functions] \label{locsubsuper} Let $0<\varepsilon<1$. Then there exists $\xb=\xb(K,N,\xm,\xe)>0$ small enough such that the followings hold in $K_\xb$.
	
	(i) If $\mu \in (-\infty,H^2)$ then
	\begin{align}
	\BAL
	\eta_{\am,\varepsilon}:=d_K^{-\am}-d_K^{-\am+\xe} &\geq 0,& \quad&  \zeta_{\am,\varepsilon}:=d_K^{-\am}+d_K^{-\am+\xe},&\\
	-L_\xm  \eta_{\am,\varepsilon}&\geq0,& \quad& -L_\xm  \zeta_{\am,\varepsilon}\leq0.&
	\EAL\label{supsol1}
	\end{align}
	
	(ii) If $\mu \in (-\infty,H^2)$ and $\xe\in (0,\min\{1,2\sqrt{H^2-\mu}\})$ then
	\begin{align}
	\BAL
	\eta_{\ap,\varepsilon}:=d_K^{-\ap}+d_K^{-\ap+\xe}  &,& \quad& \zeta_{\ap,\varepsilon}:=d_K^{-\ap}-d_K^{-\ap+\xe} \geq 0,&\\
	-L_\xm  \eta_{\ap,\varepsilon}\geq 0&,& \quad&-L_\xm  \zeta_{\ap,\varepsilon}\leq 0.&\label{supsol2}
	\EAL
	\end{align}
	
	(iii)
	\begin{align}
	\BAL
	{\xz_{\scaleto{+}{3pt},\varepsilon}}:=(-\ln d_K)d_K^{-H}-d_K^{-H+\xe} &\geq 0,& \quad&{\xz_{\scaleto{-}{3pt},\varepsilon}}:=(-\ln d_K)d_K^{-H}+d_K^{-H+\xe} ,&\\
	-L_{ H^2} {\xz_{\scaleto{+}{3pt},\varepsilon}}&\geq0,& \quad& -L_{H^2}{\xz_{\scaleto{-}{3pt},\varepsilon}}\leq 0.&
	\EAL
	\end{align}\label{subsup}
	
\end{lemma}
\begin{proof}
	For $b \neq 0$, by \eqref{laplaciand} and by taking into account that $|\nabla d_K|=1$ in $K_{3\beta_0}$, we have
	\bel{laplacianH}
	-L_\mu (d_K^b)=-(b^2+2bH+\mu) d_K^{b-2}-bd_K^{b-1}g \quad \text{in } K_{3\beta_0}.
	\ee
	In particular, since $\am$ and $\ap$ are the roots of the equation $\ga^2 - 2H\ga + \mu=0$, we obtain
	\bel{lum1}
	\BAL
	-L_\mu (d_K^{-\am}) &= \am d_K^{-(\am+1)}g, \quad \quad -L_\mu (d_K^{-\ap}) = \ap d_K^{-(\ap+1)}g,\\
	-L_\mu (d_K^{-\am+\varepsilon}) &=-\varepsilon(2H-2\am + \varepsilon)d_K^{-(\am-\varepsilon+2)} + (\am-\varepsilon)d_K^{-(\am-\varepsilon+1)}g, \\
	-L_\mu (d_K^{-\ap+\varepsilon}) &=-\varepsilon(2H-2\ap + \varepsilon)d_K^{-(\ap-\varepsilon+2)} + (\ap-\varepsilon)d_K^{-(\ap-\varepsilon+1)}g.
	\EAL
	\ee
	(i) Let $ \eta_{\am,\xe}=d_K^{-\am}-d_K^{-\am+\varepsilon}$ then by \eqref{lum1} we obtain
	\begin{align}\nonumber
	-L_\mu  \eta_{\am,\xe}
	\geq \xe(2H-2\am+\varepsilon)d_K^{-(\am-\xe+2)}-d_K^{-(\am+1)}[\am -(\am-\xe)d_K^{\varepsilon}]\|g\|_{L^\infty(K_{3\beta_0})}.
	\end{align}
	Since $\am<H$, for any $\varepsilon \in (0,1)$, there exists $\xb=\xb(K,N,\xm,\xe)$ such that $ \eta_{\am,\xe}\geq0$ in  $K_\xb$ and $-L_\mu \eta_{\am,\varepsilon} \geq 0$ in $K_\beta$.
	
	Next let $ \zeta_{\am,\varepsilon}=d_K^{-\am}+d_K^{-\am+\varepsilon}$. From \eqref{lum1}, we derive
	\begin{align*}
	-L_\mu  \xz_{\am,\xe}
	\leq -\xe(2H-2\am+\varepsilon)d_K^{-(2+\am-\xe)}+d_K^{-(2+\am-\xe)}+d_K^{-(\am+1)}[\am +(\am-\xe)d_K^{\xe}]||g||_{L^\infty(K_{3\xb_0})}.
	\end{align*}
	Since $\am<H$, we deduce that for any $0<\xe<1$ there exists $\xb=\xb(K,N,\xm,\xe)$ such that $-L_\mu \xz_{\am,\xe} \leq 0$ in $K_\beta$. \medskip
	
	(ii) We use a similar argument as above and the fact that $\ap > H$ to obtain \eqref{supsol2}. \medskip
	
	(iii) Assume that $\mu=H^2$. Let $0<\beta<\min\{\beta_0,\frac{1}{2}\}$ then $d_K(x)<\frac{1}{2}$ for any $x\in K_\beta$. Set $\tilde{\zeta}=(-\ln d_K)d_K^{-H}.$
	Then by straightforward calculations we have
	\begin{align}\nonumber
	\xD \tilde \zeta
	=-H^2(-\ln d_K )d_K^{-H-2}-(1+H(-\ln d_K))d_K^{-H-1}g.\label{ole}
	\end{align}
	By the above equality and \eqref{laplacianH}, we have
	\begin{align}
	-L_{H^2}{\xz_{\scaleto{+}{3pt},\varepsilon}}=\xe^2d^{-(2+H-\xe)}+g((1+H(-\ln d_K))d_K^{-H-1}-(H-\xe)d_K^{-(H-\xe+1)}) \geq 0,
	\end{align}
	provided $0<\xe<1$ and $\xb$ are small enough. The rest of the proof is similar to that of statements (i),(ii) and we omit it.
\end{proof}

Let $W$ and $\tilde W$ be as in \eqref{W} and \eqref{tildeW} respectively. Let us state now the boundary Harnack inequality.
\begin{lemma} \label{lemharnack}
	Let $\xb>0$ be the constant in Lemma \ref{subsup}, $\xi \in\partial \Omega \cup K$ and $0<r<\frac{\xb}{2}.$ We assume that $u\in H^1_{loc}(B_{r}(\xi)\cap(\xO\setminus K))\cap C(B_r(\xi)\cap(\xO\setminus K))$ is an $L_\xm$-subharmonic in $B_{r}(\xi)\cap(\xO\setminus K)$ and
	\bel{bdrcond} \lim_{\dist(x,F)\to 0}\frac{u_+(x)}{\tilde W(x)}=0,\quad\forall\; \text{compact} \; F\subset B_r(\xi)\cap (\partial \Omega \cup  K),
	\ee
	where $u_+=\max\{0,u\}$.
	Let $\eta$ be a smooth function such that  $0\leq\eta_r\leq1$, $\supp \eta \Subset B_{\frac{r}{2}}(\xi)$, $\eta_r=1$ on $B_{\frac{r}{4}}(\xi)$. Then
	\bel{regu1} \frac{\eta_ru_+}{\xf_\xm}\in H^1_0(\xO\setminus K;\xf_\xm^2),
	\ee
	and
	\bel{regu2}\sup_{x\in B_{\frac{r}{4}}(\xi) \cap (\Omega \setminus K)}\frac{u_+(x)}{\xf_\xm(x)}\leq C,
	\ee
	where $C=C(u,\xO,K,r)>0$. Furthermore, if $u$ is nonnegative $L_\mu$-harmonic then there exists $c=c(\Omega,K,N,k)>0$ such that
	\be \label{harnack}
	\frac{u(x)}{\xf_\xm(x)}\leq c\frac{u(y)}{\xf_\xm(y)},\quad\forall x,y\in B_{\frac{r}{16}}(\xi)\cap(\xO\setminus K).
	\ee
\end{lemma}
\begin{proof}
	
	We will only consider the case $\xm<H^2$ and $\xi \in K$, $B_r(\xi)\subset\xO$ since the proof in the other cases is very similar and we omit it.
	
	For $l>0$, set $w_l= (u-l\eta_{\ap,\varepsilon})_+$ where $\eta_{\ap,\varepsilon}$ is the supersolution constructed in \eqref{supsol1}.
	Then by Kato's inequality we deduce $-L_\mu w_l \leq 0$ in $B_r(\xi) \setminus K$.
	Set $v_l:=w_l/\phi_\mu$, then straightforward calculations lead to
	\bel{subsolution}
	-\div(\phi_\mu^2\nabla v_l)+\lambda_\mu \phi_\mu^2v_l\leq0 \quad\text{in } B_r(\xi)\setminus K.
	\ee
	We note here that $v_l=0$ if $u\leq l\eta_{\ap,\varepsilon}$, thus by the assumptions we can easily obtain that $v_l\in H^1(B_\frac{r}{2}(\xi);\phi_\mu^2)$.
	In view of the proof of \cite[Theorem 3.7]{FMoT}, we can prove the existence of a constant $r_{\xb_0}$ and $C=C(K)>0$ such that for any $r'\leq\min\{\frac{r}{2},r_{\xb_0}\}$ and $p\geq1$ the following inequality holds
	
	\bel{moser}
	\sup_{B_{\frac{r'}{2}}(\xi) \setminus K} v_l\leq C \left(\left(\int_{B_{r'}(\xi) \setminus K}\phi_\mu^2dx\right)^{-1}\int_{B_{r'}(\xi) \setminus K} |v_l|^p\phi_\mu^2dx\right)^\frac{1}{p}.
	\ee
	
	From \eqref{bdrcond} and the definition of $w_l$, we have
	$ w_l \leq u_+ \le c\tilde W = cd_K^{-\ap}$ in $B_{\frac{r}{2}}(\xi) \setminus K$.
	This and \eqref{eigenfunctionestimates} imply that
	\begin{align*}
	\int_{B_{r'}(\xi) \setminus K} |v_l|\phi_\mu^2dx &= \int_{B_{\frac{r}{2}}(\xi) \setminus K} |w_l|\phi_\mu dx  \lesssim  \int_{B_{\frac{r}{2}}(\xi) \setminus K} d_K^{-2H}dx<\infty,
	\end{align*}
	since $2H=N-k-2<N-k$.
	Thus by \eqref{moser} and the above inequality we deduce $v_l \leq C_1$ in
	$B_{\frac{r'}{2}}(\xi)$, where $C_1>0$ does not depend on $l$. Thus
	$w_l \leq C_1\phi_\mu$  in  $B_{\frac{r'}{2}}(\xi) \setminus K$.
	By letting $l \to 0$, we derive $u_+ \leq C_1\phi_\mu$ in $B_{\frac{r'}{2}}(\xi) \setminus K$.
	Thus by a covering argument we can find a constant $C_2>0$ such that
	\bel{fragma}
	u_+\leq C_2\phi_\mu \quad \text{in } B_{\frac{r}{2}}(\xi) \setminus K.
	\ee
	This implies $v_0=\frac{u_+}{\phi_\mu} <C_2$ in $B_{\frac{r}{2}}(\xi) \setminus K$.
	
	If we use $\eta^2_r v_l$ as a test function in \eqref{subsolution} we can easily obtain
	$$\int_{B_\frac{r}{2}(\xi) \setminus K}|\nabla (\eta_r v_l)|^2\xf_\xm^2dx+\xl_\xm \int_{B_\frac{r}{2}(\xi) \setminus K}|\eta_r v_l|^2\xf_\xm^2dx\leq C(|\nabla\eta_r|)\int_{B_\frac{r}{2}(\xi) \setminus K}|v_l|^2\xf_\xm^2dx.$$
	By \eqref{fragma}, letting  $l \to0$, we obtain that $\eta_r v_0 \in H^1(\Omega;\phi_\mu^2)$,  which in turn implies that $ \frac{\eta_r u_+}{\phi_\mu}\in H^1(\xO;\xf_\xm^2)=H_0^1(\Omega \setminus K;\phi_\mu^2)$.

 Finally, by combining an argument as in the proof of \cite[Proposition 2.11, page 480]{GkV} and Harnack inequality \cite[Theorem 3.7]{FMoT}, we obtain boundary Harnack inequality \eqref{harnack}.
\end{proof}

The next result provides the maximum principle for $L_\mu$-subharmonic function.

\begin{lemma} \label{comparison}
	Assume $\lambda_\mu>0$. Let $u \in H^1_{loc}(\Omega \setminus K)\cap C(\Omega \setminus K)$ be $L_\xm$-subharmonic in $\xO\setminus K$. Assume that
	\be
	\limsup_{\dist(x,F)\to 0}\frac{u(x)}{\tilde W(x)}\leq0,\quad\forall\; \text{compact} \; F\subset \partial \Omega \cup K.\label{boundary0}
	\ee
	Then $u\leq0$ in $\xO\setminus K.$
\end{lemma}
\begin{proof}
	First we note that $u_+=\max\{u,0\} \in H^1_{loc}(\Omega \setminus K)\cap C(\Omega \setminus K)$ and it is a nonnegative $L_\mu$-subharmonic function in $\Omega \setminus K$. Moreover, $u_+$ satisfies \eqref{bdrcond}. Set $w=u_+/\phi_\mu$ then from \eqref{regu1}, we deduce that $w \in H^1_0(\Omega \setminus K;\phi_\mu^2)$. By straightforward calculations we have
	\bel{su0}
	-\div(\xf_\xm^2\nabla w)+\lambda_\mu \phi_\mu^2 w \leq 0 \quad \text{in } \Omega \setminus K.
	\ee
	Since $v \in H^1_0(\Omega \setminus K;\phi_\mu^2)$, we can use it as a test function for \eqref{su0} and obtain
	$$\int_{\Omega \setminus K}|\nabla w|^2\phi_\mu^2dx+\lambda_\mu \int_{\Omega \setminus K}|w|^2\phi_\mu^2dx\leq 0.$$
	Since $\lambda_\mu>0$, we deduce $w=0$ and hence the result follows straightforward.
\end{proof}

\section{Green kernel}
In this section we prove the existence and sharp two-sided estimates of the Green kernel of $-L_\mu$ in $\Omega \setminus K$. Hereinafter, we assume that $0 \leq k<N-2$, $\mu \leq H^2$ and $\lambda_\mu >0$.

\begin{proposition}[Existence and two-sided estimates of Green kernel] \label{green}  For any $y \in \Omega \setminus K$, there exists a minimal Green kernel  $G_{\mu}(\cdot,y)$ of $-L_\mu$ in $\Omega \setminus K$, i.e. $G_\mu(\cdot, y)$ is the minimal solution of \eqref{greeneq} in the sense of distributions. Furthermore, the following estimates hold.
	
	(i) If $\mu< \left( \frac{N-2}{2}\right)^2$ then for any $x,y \in \Omega \setminus K$, $x \neq y$,
	\bel{Greenest1a}
	G_\mu(x,y) \approx |x-y|^{2-N} \left(1 \wedge \frac{d(x)d(y)}{|x-y|^2}\right)  \left(\frac{|x-y|}{d_K(x)}+1\right)^\am
	\left(\frac{|x-y|}{d_K(y)}+1\right)^\am.
	\ee
	
	(ii) If $k=0$, $K=\{0\}$ and $\mu= \left( \frac{N-2}{2}\right)^2$ then for any $x,y \in \Omega \setminus K$, $x \neq y$,
	\bel{Greenest1b} \BAL
	G_\mu(x,y) &\approx |x-y|^{2-N} \left(1 \wedge \frac{d(x)d(y)}{|x-y|^2}\right) \left(\frac{|x-y|}{|x|}+1\right)^{\frac{N-2}{2}}
	\left(\frac{|x-y|}{|y|}+1\right)^{\frac{N-2}{2}} \\
	&\quad +(|x||y|)^{-\frac{N-2}{2}}\left|\ln\left(1 \wedge \frac{|x-y|^2}{d(x)d(y)}\right)\right|.
	\EAL \ee

\end{proposition}

\begin{proof} \textit{Existence.} We see from \cite[Proposition 2.8]{FMoT} that there exists a heat kernel, denoted by $h(t,x,y)$, associated to $\partial_t-L_\mu$. By \cite[Theorem 1.3]{FMoT}, for any $u_0 \in L^2(\Omega)$, there exists a unique solution $u(t,x)$ of
	$$ \left\{ \BAL
	\partial_t u-L_\xm u&=0  \quad &&\text{in } (0,\infty)\times (\Omega \setminus K)\\
	u(t,x)&=0 &&\text{in } (0,\infty)\times \partial(\Omega \setminus K) \\
	u(0,x)&=u_0(x) &&\text{in } \Omega \setminus K.
	\EAL \right. $$
	Furthermore, $u(t,x)=\int_0^th(s,x,y)u_0(y)dy.$
	
	Next, by applying \cite[Theorem 1.3]{FMoT} $d_1=d,\;d_{N-k}=d_K,$ $\xa_1=1,$  $\xa_{N-k}=-\am$ and $\xa_i=0$ for all $i\neq1,N-k$, we deduce that there exist positive constants $c_1<c_2$ and $T$ depending on $N,\Omega,K,\mu$ such that the following estimates are valid
	\begin{equation} \label{heatest1a} \BAL &c_1\left(\frac{\sqrt{t}}{d(x)}+1\right)^{-1}\left(\frac{\sqrt{t}}{d(y)}+1\right)^{-1}
	\left(\frac{\sqrt{t}}{d_K(x)}+1\right)^\am\left(\frac{\sqrt{t}}{d_K(y)}+1\right)^\am t^{-\frac{N}{2}}e^{-\frac{c_2|x-y|^2}{t}}\leq h(t,x,y)\\
	&\leq c_2\left(\frac{\sqrt{t}}{d(x)}+1\right)^{-1}\left(\frac{\sqrt{t}}{d(y)}+1\right)^{-1}
	\left(\frac{\sqrt{t}}{d_K(x)}+1\right)^\am\left(\frac{\sqrt{t}}{d_K(y)}+1\right)^\am t^{-\frac{N}{2}}e^{-\frac{c_1|x-y|^2}{t}},
	\EAL \end{equation}
	for all $t \in (0,T),\; x,y\in \xO\setminus K$ and
	\begin{equation} \label{heatest2a}
	c_1\frac{d(x)d(y)}{d_K(x)^{\am}d_K(y)^{\am}}e^{-\xl_\xm t}\leq h(t,x,y) \leq  c_2\frac{d(x)d(y)}{d_K(x)^{\am}d_K(y)^{\am}}e^{-\xl_\xm t},
	\end{equation}
	for all $t\geq T,\; x,y\in \xO\setminus K.$
	
	Set
	$$ G_{\mu}(x,y)=\int_0^\infty h(t,x,y)dt, \quad x,y \in \Omega \setminus K,\; x \neq y.
	$$
		Since $\lambda_\mu>0$, by the standard argument and \eqref{heatest1a} and \eqref{heatest1a}, we can show that $G_\mu$ is a Green kernel of $-L_\mu$ in $\Omega \setminus K$. \smallskip

	\textit{Minimality.} Assume $u$ is a nonnegative solution of \eqref{greeneq} in the sense of distributions.  Set $v=G_\mu(\cdot,y)-u$ then $v$ is an $L_\mu$-harmonic function in $\Omega \setminus K$ and hence by the standard elliptic theory, $v \in H^1_{loc}(\Omega \setminus K)\cap C(\Omega \setminus K)$. Moreover $v \leq C_y \ei$ in $\Omega \setminus K$. Therefore $v_+=\max\{v,0\}$ is a nonnegative $L_\mu$-subharmonic function and $v_+ \in  H^1_{loc}(\Omega \setminus K)\cap C(\Omega \setminus K)$ and $0 \leq v_+ \leq C_y \ei$ in $\Omega \setminus K$. This implies $v_+$ satisfies \eqref{bdrcond}. Since $\lambda_\mu>0$, using Lemma \ref{comparison}, we derive $v_+ = 0$ and hence $G_\mu(x,y) \leq u(x)$ for all $x \in \Omega \setminus K$.
	
	\medskip
	
	\textit{Estimate on $G_\mu$.}

From \eqref{heatest1a}, \eqref{heatest2a},  we can show that there exist $C_i=C_i(\Omega,K,\mu,N)>0$, $i=1,2$, such that

\begin{align}
\BAL &C_1 \left( 1 \wedge \frac{d(x)d(y)}{t} \right)
\left(\frac{\sqrt{t}}{d_K(x)}+1\right)^\am\left(\frac{\sqrt{t}}{d_K(y)}+1\right)^\am t^{-\frac{N}{2}}e^{-\frac{C_2|x-y|^2}{t}}\leq h(t,x,y)\\
&\leq C_2\left( 1 \wedge \frac{d(x)d(y)}{t} \right)
\left(\frac{\sqrt{t}}{d_K(x)}+1\right)^\am\left(\frac{\sqrt{t}}{d_K(y)}+1\right)^\am t^{-\frac{N}{2}}e^{-\frac{C_1|x-y|^2}{t}},\\
\EAL\label{heatest1}
\end{align}
for all $t \in (0,T),\; x,y\in \xO\setminus K$ and
\begin{align}
\BAL
C_1\frac{d(x)d(y)}{d_K(x)^{\am}d_K(y)^{\am}}e^{-\xl_\xm t}\leq h(t,x,y) \leq  C_2\frac{d(x)d(y)}{d_K(x)^{\am}d_K(y)^{\am}}e^{-\xl_\xm t},\EAL\label{heatest2}
\end{align}
for all $t\geq T,\, x,y\in \xO\setminus K$.

We write
\bel{Gh}
G_{\mu}(x,y)= \int_0^Th(t,x,y)dt+\int_T^\infty h(t,x,y)dt.
\ee

To deal with the first term on the right hand side of \eqref{Gh}, we use \eqref{heatest1}. By change of variable $s=\frac{|x-y|^2}{t}$, we obtain for $i=1,2$,
\begin{equation} \label{2} \begin{aligned}
&\int_0^T\left( 1 \wedge \frac{d(x)d(y)}{t} \right)\left(\frac{\sqrt{t}}{d_K(x)}+1\right)^\am \left(\frac{\sqrt{t}}{d_K(y)}+1\right)^{\am}
t^{-\frac{N}{2}}e^{-\frac{C_i|x-y|^2}{t}}dt \\
&=|x-y|^{-(N-2)} \int_{\frac{|x-y|^2}{T}}^\infty I_i(s)ds,
\end{aligned} \end{equation}
where
$$ I_i(s)=\left( 1 \wedge s\frac{d(x)d(y)}{|x-y|^2} \right)
\left(\frac{|x-y|}{\sqrt{s}d_K(x)}+1\right)^\am \left(\frac{|x-y|}{\sqrt{s}d_K(y)}+1\right)^{\am}s^{\frac{N}{2}-2}e^{-C_i s}.
$$
Put $D= \diam(\Omega)^2$. By straighforward calculations, we obtain, for any $i=1,2$,
\begin{align} \label{4} \BAL
\int_{1 \lor \frac{2D}{T}}^\infty I_i(s) ds
\approx \left(1 \wedge \frac{d(x)d(y)}{|x-y|^2} \right) \left(\frac{|x-y|}{d_K(x)}+1\right)^{\am} \left(\frac{|x-y|}{d_K(y)}+1\right)^{\am}.
\EAL \end{align}

We will only consider the case $k=0$, $K=\{0\}$ and $\mu=\left( \frac{N-2}{2} \right)^2$ since the proof in the other cases is similar. In this case $\am=\frac{N-2}{2}$ and $d_K(x)=|x|$.

For the lower bound, we write
\begin{align} \label{3}
\BAL
&\int_{\frac{|x-y|^2}{T}}^\infty I_2(s)ds
=\int_{\frac{|x-y|^2}{T}}^{1 \lor \frac{2D}{T}} I_2(s)ds + \int_{1 \lor \frac{2D}{T}}^\infty I_2(s)ds.
\EAL
\end{align}
The first term on the right hand side of \eqref{6} can be estimated from below as
\begin{align}\BAL
\int_{\frac{|x-y|^2}{T}}^{1 \lor \frac{2D}{T}} I_2(s)ds
& \gtrsim \left(\frac{|x-y|^2}{|x||y|}\right)^{\frac{N-2}{2}}\int_{\frac{|x-y|^2}{T}}^{1 \lor \frac{2D}{T}} \left(1 \wedge s\frac{d(x)d(y)}{|x-y|^2} \right)
s^{-1}ds\\
& \gtrsim \left(\frac{|x-y|^2}{|x||y|}\right)^{\frac{N-2}{2}}\left|\ln\left(1 \wedge \frac{|x-y|^2}{d(x)d(y)} \right)\right|.
\EAL\label{6}
\end{align}
Thus the lower bound in \eqref{Greenest1b} follows from \eqref{heatest1}, \eqref{2}--\eqref{6}.

For the upper bound, we have
\bel{7}
\BAL
&\int_{\frac{|x-y|^2}{T}}^{1 \lor \frac{2D}{T}} I_1(s)ds
\lesssim (|x||y|)^{-\frac{N-2}{2}}\int_{\frac{|x-y|^2}{T}}^{1 \lor \frac{2D}{T}}  \left (1 \wedge s\frac{d(x)d(y)}{|x-y|^2} \right)
|x-y|^{N-2} s^{-1}e^{-C_1 s}ds\\
&+ (|x||y|)^{-\frac{N-2}{2}}\int_{\frac{|x-y|^2}{T}}^{1 \lor \frac{2D}{T}} \left (1 \wedge s\frac{d(x)d(y)}{|x-y|^2} \right)
\left(\sqrt{s}|x-y|(|x|+|y|)+s|x||y|\right)^{\frac{N-2}{2}}s^{-1}e^{-C_1 s}ds.
\EAL
\ee
Note that
\begin{align} \nonumber
&\int_{\frac{|x-y|^2}{T}}^{1 \lor \frac{2D}{T}} \left(1 \wedge s\frac{d(x)d(y)}{|x-y|^2} \right) s^{-1}e^{-C_1 s}ds \leq \int_{\frac{|x-y|^2}{T}}^{1 \lor \frac{2D}{T}} \left (1 \wedge s\frac{d(x)d(y)}{|x-y|^2} \right) s^{-1}ds\\
&\lesssim   \left(1 \wedge \frac{d(x)d(y)}{|x-y|^2} \right) +  \left(1 \wedge \frac{d(x)d(y)}{|x-y|^2} \right) \left|\ln\left(1 \wedge \frac{|x-y|^2}{d(x)d(y)} \right)\right|. \label{8}
\end{align}
We can also estimate
\begin{align} \label{9}
\BAL
&\int_{\frac{|x-y|^2}{T}}^{1 \lor \frac{2D}{T}} \left(1 \wedge s\frac{d(x)d(y)}{|x-y|^2} \right)
\left(\sqrt{s}|x-y|(|x|+|y|)+s|x||y|\right)^{\frac{N-2}{2}}s^{-1}e^{-C_1 s}ds \\
&\lesssim \left(1 \wedge \frac{d(x)d(y)}{|x-y|^2} \right)
\left(\left(|x-y|+|x|\right)\left(|x-y|+|y|\right)\right)^{\frac{N-2}{2}}.\EAL
\end{align}
Combining \eqref{heatest1}, \eqref{heatest2}, \eqref{2}--\eqref{3}, \eqref{7}--\eqref{9} yields the upper bound of \eqref{Greenest1b}.
\end{proof} \smallskip

\begin{proof}[\textbf{Proof of Theorem \ref{Greenkernel}}] The existence of the minimal Green kernel is proved in Proposition \ref{green}. The estimates \eqref{Greenesta} and \eqref{Greenestb} follows from \eqref{Greenest1a}, \eqref{Greenest1b} and the estimate
	\begin{align} \label{12}
	\frac{d_K(x)d_K(y)}{(|x-y|+d_K(x))(|x-y|+d_K(y))}\approx 1 \wedge \frac{d_K(x)d_K(y)}{|x-y|^2}.
	\end{align}
	The proof is complete.
\end{proof}

\section{Linear equations}

The aim of this subsection is to prove the existence and uniqueness of the solution of $-L_\mu u=f$ with prescribed smooth boundary data. Let us first define the weak solutions.

\begin{definition} \label{weaksolOK}
	Let $f\in L_{loc}^1(\Omega \setminus K)$. We say that  $u$ is a weak solution of equation
	\bel{LE1} -L_\mu u=f \quad \text{in }  \Omega \setminus K
	\ee
	if  $u \in H^1_{loc}(\Omega\setminus K)$ and $u$ satisfies
	\be \label{weaksolution}
	\int_{\Omega} \nabla u \cdot \nabla\psi dx-\xm \int_{\Omega} \frac{u\psi}{d^2_K} dx=\int_{\Omega} f\psi dx\quad\forall \psi \in C_c^1(\Omega \setminus K).
	\ee
\end{definition}

In the next lemma we give the first existence and uniqueness result.

\begin{lemma} \label{existence1}
	For any $f\in L^2(\xO)$ there exists a unique weak solution $u$ of \eqref{LE1} such that $\phi_\mu^{-1}u \in H^1(\Omega; \phi_\mu^2)$. Furthermore, there holds
	\bel{weaksolutionest}
	\int_{\xO} |\nabla u|^2 dx-\xm \int_{\Omega} \frac{u^2}{d^2_K}dx\lesssim \int_{\Omega} |f|^2dx.
	\ee
\end{lemma}
\begin{proof} We first observe that $u$ is a weak solution of \eqref{LE1} if and only if $v=\frac{u}{\phi_\mu}$ satisfies
	\bel{lem}
	\int_{\Omega} \phi_\mu^2\nabla v \cdot \nabla\zeta dx+\lambda_\mu \int_{\Omega} \phi_\mu^2 v\zeta dx=\int_{\Omega} \phi_\mu  f\zeta dx
	\ee
	for any $\zeta \in  C_c^1(\Omega \setminus K)$.
	
	We define the inner product $\langle, \rangle$ and the functional $T_f$ on $H_0^1(\Omega \setminus K;\phi_\mu^2)$ respectively by
	$$ \langle \zeta, \tilde \zeta \rangle = \int_{\Gw \setminus K} \phi_\mu^2(\nabla \zeta \cdot \nabla \tilde \zeta + \lambda_\mu \zeta \tilde \zeta)dx, \quad \zeta, \tilde \zeta \in H_0^1(\Omega \setminus K;\phi_\mu^2),
	$$
	$$T_f(\zeta)=\int_{\Omega \setminus K} \phi_\mu f \zeta dx, \quad \zeta \in H_0^1(\Omega \setminus K;\phi_\mu^2). $$
	We see, by using H\"older inequality and the fact that $f \in L^2(\Omega)$,  that $T_f$ is a bounded linear functional on $H_0^1(\Omega \setminus K;\phi_\mu^2)$. Therefore by Riesz's representation theorem, there exists a unique function $v \in H_0^1(\Omega \setminus K;\phi_\mu^2)$ satisfying
	\bel{lem1}
	\langle v,\zeta \rangle = T_f(\zeta) \quad \forall \zeta \in H_0^1(\Omega \setminus K;\phi_\mu^2).
	\ee
	Furthermore, by choosing $\zeta=v$ in \eqref{lem1} and then using Young's inequality, we obtain
	\bel{lem2} \int_{\Omega} \phi_\mu^2|\nabla v|^2 dx+\lambda_\mu \int_{\Omega} \phi_\mu^2 v^2 dx\lesssim \int_{\Omega} |f|^2 dx.
	\ee

	By Proposition \ref{test} and \eqref{eigenfunctionestimates}, we see that $v \in H^1(\Omega; \phi_\mu^2)$.  Putting $u=\ei v$, we obtain from the above observation that $u$ satisfies \eqref{weaksolution}.
	
	Conversely, by the uniqueness of $v$ and the standard density argument, we see that every weak solution $u$ of \eqref{LE1} such that $\phi_\mu^{-1}u \in H^1(\Omega;\phi_\mu^2)$ can be constructed in this way, and hence the uniqueness for \eqref{LE1} follows.
	Finally, \eqref{weaksolutionest} follows from \eqref{lem2}.
\end{proof}

In the following lemma we prove the existence, as well as pointwise estimates, of the solutions for the equation  $-L_\mu u=f$, with ``zero boundary data".

\begin{lemma} \label{existence1b} Let $f\in L^\infty(\xO)$. Then there exists a unique solution $u$  of \eqref{LE1} such that $\phi_\mu^{-1}u \in H^1(\Omega; \phi_\mu^2)$ and
	\be \label{boundary=0}
	\lim_{\dist(x,F)\to 0}\frac{u(x)}{\tilde W(x)}=0,\quad\forall\; \text{compact} \; F\subset \partial \Omega \cup K.
	\ee
	The solution can be written as $u=\BBG_\mu[f]$  and there holds
	\bel{veryweakest0}
	|u(x)|\lesssim ||f||_{L^\infty(\xO)}d(x)d_K(x)^{-\am},\quad \forall x\in \xO\setminus K.
	\ee
\end{lemma}

\begin{proof}
	By Lemma \ref{existence1}, there exists a unique solution $u$ of \eqref{LE1} such that $\phi_\mu^{-1}u \in H^1(\Omega; \phi_\mu^2)$. Furthermore, by the standard argument, we can show that $u=\BBG_\mu[f]$.
	
	Next if we put $v=u/\ei$ then $v \in  H^1(\Omega; \phi_\mu^2)$ and it satisfies \eqref{lem}. Since $f \in L^\infty(\Omega)$, by  using a Moser iteration argument similar to the one in \cite[Theorem 2.12]{FMoT2} (see also \cite[Theorem 3.7]{FMoT}) we can show that there exists a positive constant $C>0$ depending on $N,\Omega,K,\mu,\| f\|_{L^\infty(\Omega)}$ such that $\sup_{x\in\xO\setminus K}|v(x)|\leq C$, which implies $u(x)|\leq C\ei(x)$ for every $x\in\Omega \setminus K$. This in turn yields \eqref{veryweakest0} due to \eqref{eigenfunctionestimates}. Combining \eqref{veryweakest0} and \eqref{W} yields \eqref{boundary=0}.
\end{proof}

\begin{remark} If we choose $f=1$ then we derive from \eqref{veryweakest0} and \eqref{eigenfunctionestimates} that
	\bel{G1<eigen} \BBG_\mu[1](x) \lesssim \phi_\mu(x) \quad \forall x \in \Omega \setminus K.
	\ee
\end{remark}

In order to treat more general data, we need the following result.

\begin{lemma} \label{anisotita}
	Let $0<\xa<\xg<N$ and $\alpha<N-k.$ Then
	\begin{equation} \label{eq:ani}
	\sup_{z \in \overline \Omega}\int_{\Omega \setminus K} |y-z|^{-N+\gamma}d_K(y)^{-\alpha}dy \lesssim 1.
	\end{equation}
	The implicit constant in \eqref{anisotita} depends on $N,\Omega,K,\alpha$ and $\gamma$.
\end{lemma}
\begin{proof} First we consider $z \in K_{\beta_1}$ where $\beta_1$ is the constant in\eqref{cover}. We write
	\begin{equation} \label{anis1} \int_{\Omega \setminus K} |y-z|^{-N+\gamma}d_K(y)^{-\alpha}dy = \int_{\Omega \setminus K_{\beta_1} } |y-z|^{-N+\gamma}d_K(y)^{-\alpha}dy + \int_{K_{\beta_1} } |y-z|^{-N+\gamma}d_K(y)^{-\alpha}dy.
	\end{equation}
	Since $\alpha>0$ and $\gamma>0$, we have
	\begin{equation} \label{anis2} \int_{\Omega \setminus K_{\beta_1}} |y-z|^{-N+\gamma}d_K(y)^{-\alpha}dy < \beta_1^{-\alpha} \int_{\Omega \setminus K_{\beta_1} } |y-z|^{-N+\gamma}dy \lesssim 1.
	\end{equation}
	Next, we deal with the second term in \eqref{anis1}. For $\lambda>0$ and $x \in K_{\beta_1}$, set
	\begin{align*}
	A_\lambda(x):=\Big\{y \in K_{\beta_1} \setminus\{x\}:\;\; |x-y|^{-N+\gamma}>\lambda \Big \},\quad
	m_{\lambda}(x):=\int_{A_\lambda(x)}d_K(y)^{-\alpha}dy.
	\end{align*}
	For $x \in K_{\beta_1}$, there is a unique $\xi \in K$  satisfies $|x-\xi|=d_K(x) \leq \beta_1$. Furthermore, there exist $C^2$ functions $\Gamma_i^\xi:\mathbb{R}^{k}\rightarrow\mathbb{R},$ $i=k+1,...,N$, such that (upon relabeling and reorienting the coordinate axes if necessary) \eqref{straigh} holds.
	
	We write
	\begin{equation} \label{m1}
	m_{\lambda}(x)=\int_{A_\xl(x)\cap\{|x-y|<d_K(y)\}}d_K(y)^{-\alpha}dy+\int_{A_\xl(x)\cap\{|x-y|\geq d_K(y)\}}d_K(y)^{-\alpha}dy.
	\end{equation}
	Since $A_\lambda(x) \subset B(x,\lambda^{-\frac{1}{N-\gamma}})$, the first term on the right hand side of \eqref{m1} is estimated as
	\begin{equation} \label{m2} \begin{aligned}
	\int_{A_\xl(x)\cap\{|x-y|<d_K(y)\}}d_K(y)^{-\alpha}dy \leq \int_{B(x,\lambda^{-\frac{1}{N-\gamma}})}|x-y|^{-\alpha}dy
	\lesssim \lambda^{-\frac{N-\alpha}{N-\gamma}}.
	\end{aligned} \end{equation}
	
	Next we treat the second term on the right hand side of \eqref{m1}. From \eqref{BinV}, if $\lambda \geq \beta_1^{-(N-\gamma)}$ then $A_\lambda(x) \subset B(x,\beta_1) \subset V(\xi,\beta_0)$. Consequently, by \eqref{propdist}, $\delta_K^\xi(y) \leq C \| K \|_{C^2}d_K(y)$ for every $y \in A_\lambda(x)$, where $\delta_K^\xi$ be defined in \eqref{dist2}. Therefore
	\bel{m3} \BAL
	\int_{A_\lambda(x)\cap\{|x-y|\geq d_K(y)\}}d_K(y)^{-\alpha}dy
	&\lesssim \int_{A_\xl(x)\cap\{|x-y|\geq d_K(y)\}}\delta_K^\xi(y)^{-\alpha}dy\\
	&\lesssim \int_{\{|\zeta''|<c \xl^{-\frac{1}{N-\xg}}\}}|\zeta''|^{-\xa}\int_{\{|\zeta'|< \xl^{-\frac{1}{N-\xg}}\}}d\zeta'd\zeta''\\
	&\lesssim \lambda^{-\frac{N-\alpha}{N-\gamma}}.
	\EAL \ee
	Here in the second estimate we have used the change of variable $\zeta'=y'$ and  $\zeta_i=y_i-\Gamma_i^\xi(y'),\;\forall i=k+1,...,N$. Combining \eqref{m1}--\eqref{m3} yields
	\begin{equation} \label{m4} m_\lambda(x) \lesssim \lambda^{-\frac{N-\alpha}{N-\gamma}}
	\end{equation}
	for all $\lambda \geq \beta_1^{-(N-\gamma)}$. Since $\alpha<N-k$, we can deduce that \eqref{m4} holds true for all $\lambda >0$.
	
	Therefore we can apply \cite[Lemma 2.4]{BVi} with $\mathcal{H}(x,y)=|x-y|^{-N+\xg}$,  $\Omega=K_{\beta_1}$, $\eta=d_K^{-\xa}$ and $\omega=\delta_z$,
	where $\delta_z$ is the Dirac measure concentrated at $z$, to obtain
	$$ \| \mathcal{H}(z,\cdot)  \|_{L_w^{\frac{N-\alpha}{N-\gamma}}(K_{\beta_1};d_K^{-\alpha}) } \lesssim 1,
	$$
	where $L_w^p$ ($p>1$) denotes the weak $L^p$ space or Marcikiewicz space (see e.g. \cite{MVbook}). Hence
	$$
	\int_{K_{\beta_1}} |y-z|^{-N+\gamma}d_K(y)^{-\alpha}dy  \lesssim 1.
	$$
	This, together with \eqref{anis1} and \eqref{anis2}, yields
	\begin{equation} \label{anis4}
	\int_{\Omega \setminus K} |y-z|^{-N+\gamma}d_K(y)^{-\alpha}dy \lesssim 1 \quad \forall z \in K_{\beta_1}.
	\end{equation}
	
	If $z\in {\overline \Omega} \setminus K_{\beta_1}$ then
	\bel{kan2} \BAL
	&\int_{\Omega \setminus K}|y-z|^{-N+\gamma}d_K(y)^{-\alpha}dy \\
	&\leq \left(\frac{2}{\beta_1}\right)^{N-\gamma} \int_{K_{\frac{\beta_1}{2}}}d_K(y)^{-\alpha}dy +  \left(\frac{2}{\beta_1}\right)^\xa\int_\xO|y-z|^{-N+\xg}dx \lesssim 1.
	\EAL \ee
 Combining \eqref{anis4} and \eqref{kan2} leads to \eqref{eq:ani}.
\end{proof}

\begin{lemma} \label{existence2} Let $0<b<\ap+2$ and $f\in L^\infty(\xO)$. Then there exists a unique solution $u\in H^1_{loc}(\xO\setminus K)\cap C(\xO\setminus K)$ satisfying \eqref{boundary=0} of
	\bel{dKb}  -L_\xm u=fd_K^{-b} \quad \text{in } \Omega \setminus K,
	\ee
	in the sense of Definition \ref{weaksolOK}. Moreover, for any $\xg\in[\am,\infty)\cap(b-2,\infty)$, there holds
	\bel{veryweakest}
	|u(x)|\lesssim ||f||_{L^\infty(\xO)}d(x)d_K(x)^{-\xg},\quad x\in \xO\setminus K.
	\ee
	Here the implicit constant in \eqref{veryweakest} depends on $N,k,\xm,\xO,K,b$.

\end{lemma}
\begin{proof} We consider only the case  $0<\mu<H^2$ since the proof in other cases is similar with some minor modifications.  In this case $\am>0$.
	
	Assume that $f\geq0.$ Set $f_n=\min\{fd_K^{-b},n\}$. Then by Lemma \ref{existence1b}, there exists a unique solution $u_n$ of $-L_\xm v=f_n$ in $\xO\setminus K$ satisfying \eqref{boundary=0}. Morever this solution can be written as $u_n=\BBG_\mu[f_n]$. 	
	Let $x \in K_{\beta_0}$ where $\beta_0$ is the constant \eqref{straigh}. By \eqref{Greenest1a}, we have
	\begin{align*}
	u_n(x)
	&\lesssim \int_{\Omega \setminus K} |x-y|^{-N+2}\left(\left(\frac{|x-y|}{d_K(x)}\right)^\am+1\right)
	\left(\left(\frac{|x-y|}{d_K(y)}\right)^\am+1\right) f_n(y)dy\\
	&= J_1 + J_2 + J_3 + J_4,\\
	\end{align*}
	where
	\begin{align*}
	J_1 &= d_K(x)^{-\am}\int_{\Omega \setminus K} |x-y|^{-N+2+2\am}d_K(y)^{-\am} f_n(y)dy, \\
	J_2 &= d_K(x)^{-\am}\int_{\Omega \setminus K} |x-y|^{-N+2+\am} f_n(y)dy, \\
	J_3 &= \int_{\Omega \setminus K} |x-y|^{-N+2+\am}d_K(y)^{-\am} f_n(y)dy, \\
	J_4 &= \int_{\Omega \setminus K} |x-y|^{-N+2}f_n(y)dy.
	\end{align*}
	
	First we note that if $d_K(y)\leq \frac{1}{4}d_K(x)$ then $|x-y|\geq\frac{3}{4}d_K(x).$ Thus, since $0<b<\ap +2$, we have, for any $\gamma \in [\am,\infty) \cap (b-2,\infty)$,
	\begin{align*}
	J_1 &\leq ||f||_{L^\infty(\Omega)}d_K(x)^{-\am} \int_{\Omega \setminus K} |x-y|^{-N+2+2\am}d_K(y)^{-\am-b}dy\\
	&\lesssim ||f||_{L^\infty(\Omega)} d_K(x)^{-\xg}\int_{(\Omega \setminus K)  \cap\{d_K(y)\leq \frac{1}{4}d_K(x)\}} |x-y|^{-N+2+\am+\xg}d_K(y)^{-\am-b}dy\\
	&\quad + ||f||_{L^\infty(\Omega)} d_K(x)^{-\xg}\int_{(\Omega \setminus K)  \cap\{d_K(y)\geq \frac{1}{4}d_K(x)\}} |x-y|^{-N+2+2\am}d_K(y)^{-2\am-b+\xg}dy\\
	&\lesssim ||f||_{L^\infty(\Omega)} d_K(x)^{-\xg},
	\end{align*}
	where in the last inequality we have used Lemma \ref{anisotita}.
	
	Similarly, since $0<b<\ap +2$, for any $\gamma \in [\am,\infty) \cap (b-2,\infty)$, we can show that
	\begin{align*}
	J_2 + J_3 + J_4 \lesssim d_K(x)^{-\gamma}.
	\end{align*}

	Combining the above estimates, we deduce that for any $\gamma\in[\alpha_+,\infty)\cap(b-2,\infty)$,
	\bel{veryweakest2a}
	0 \leq u_n(x) \lesssim ||f||_{L^\infty(\Omega)}d_K(x)^{-\xg},\quad \forall x\in K_{\beta_0}.
	\ee
	This implies that
	\be
	0 \leq u_n(x) \lesssim ||f||_{L^\infty(\Omega)}d(x)d_K(x)^{-\xg},\quad \forall x \in \xO\setminus K,\label{veryweakest2}
	\ee
	where the implicit constant depends on $N,\mu,\Omega,K,b,\gamma$.
	
	By \eqref{veryweakest2} and Lemma \ref{comparison}, $u_n\nearrow u$ locally uniformly in $\xO\setminus K$ and in $H^1_{loc}(\xO\setminus K).$ Furthermore, by the standard elliptic theory $u\in C^1(\xO\setminus K),$ and satisfies
	\be\label{veryweakest3}
	0 \leq u(x) \lesssim ||f||_{L^\infty(\xO \setminus K)}d(x)d_K(x)^{-\gamma} \quad \forall x\in \xO\setminus K.
	\ee
	This implies
	$$ 0 \leq \frac{u(x)}{\tilde W(x)} \lesssim \| f \|_{L^\infty(\Omega)}d(x)d_K(x)^{\ap - \gamma} \quad \forall x\in \xO\setminus K.
	$$
	Therefore, by choosing $\xg\in [\am,\ap)\cap(b-2,\ap)$, we derive \eqref{boundary=0}.
	
	For the general case, let $u_1$ and $u_2$ are solutions of \eqref{dKb} in $\Omega \setminus K$ with $f$ replaced by $f_-$ and $f_+$ respectively. Then $u_1$ (resp. $u_2$) satisfies \eqref{boundary=0} and  \eqref{veryweakest} with $f$ replaced by $f_-$ (resp. by $f_+$). Set $u=u_2 -u_1$ then  $u$ is a solution of \eqref{dKb} and satisfies \eqref{boundary=0}, \eqref{veryweakest}.
	
	The uniqueness follows from \eqref{boundary=0} and Lemma \ref{comparison}. Thus the proof in the case $0<\mu<H^2$ is completed.
\end{proof}

The following lemma is the main result of this subsection.
\begin{lemma} \label{mainlemma1}
	For any  $h \in C(\partial \Omega \cup K)$,
	there exists a unique $L_{\xm}$-harmonic function $u \in H^1_{loc}(\xO\setminus K)\cap C(\xO\setminus K)$  satisfying
	\bel{bdrcond2} \lim_{x\in\Omega \setminus K,\;x\rightarrow y\in\partial \Omega \cup  K}\frac{u(x)}{\tilde W(x)}=h(y)\qquad\text{uniformly w.r.t. } y\in\partial\xO\cup K.
	\ee
	Furthermore,
	\bel{ektimisi}
	\left|\!\left|\frac{u}{\tilde W}\right|\!\right|_{L^\infty(\Omega)}\lesssim ||h||_{C( \partial\Omega \cup K)}.
	\ee
\end{lemma}
\begin{proof}
	The uniqueness is a consequence of Lemma \ref{comparison}. \smallskip
	
	\emph{Existence.} First we assume that $h\in C^2(\overline{\xO})$. If $u\in  C^2(\Omega \setminus K)$ is an $L_\mu$-harmonic function in $\Omega \setminus K$ then $v=u-\tilde W h$ satisfies
	\be \label{10}
	-L_\xm v=-L_\xm (\tilde Wh)=h(-L_\xm\tilde W) -2\nabla \tilde W\nabla h-\tilde W\xD h \quad \text{in } \Omega \setminus K.
	\ee
	Thus it is enough to find a solution of \eqref{10}.
	
	In view of the proof of Lemma \ref{subsup}, we derive
	$|L_\xm \tilde W|\lesssim d_K^{-1}\tilde W$ in $\Omega \setminus K$.
	Hence we can write \eqref{10} as follows
	\bel{10'}
	-L_\mu v = fd_K^{-(\ap+1)} \quad \text{in } \Omega \setminus K,
	\ee
	for some $f \in L^\infty(\Omega)$ with $\| f\|_{L^\infty(\Omega)} \lesssim \|h\|_{C^2(\overline \Omega)}$. By Lemma \ref{existence2}, there exists a unique solution $v$ of \eqref{10} that satisfies
	\be\label{veryweakest5}
	|v(x)|\lesssim ||h||_{C^2(\overline{\xO})}d(x)d_K(x)^{-\xg} \quad \forall x\in \xO\setminus K
	\ee
	for any $\gamma \in [\am,\ap) \cap (\ap-1,\ap)$.
	Thus
	\be \label{tofragma}
	\left|\frac{u(x)}{\tilde W(x)}-h(x)\right|\lesssim ||h||_{C^2(\overline{\xO})}d(x)d_K(x)^{\ap-\xg} \quad \forall x\in \xO\setminus K.
	\ee
	This implies \eqref{bdrcond2} and \eqref{ektimisi}.
	
	If $h\in C(\partial\xO\cup K)$ then we can find a sequence $\{h_n\}_{n=1}^\infty$ of smooth functions in $\partial\xO\cup K$ such that $h_n\rightarrow h$ in $L^\infty(\partial\xO\cup K).$ Then we can find a function $H_m\in C^2(\overline{\xO})$
	with value $h_n$ on $\partial\xO\cup K,$ and $||H_n||_{L^\infty(\overline{\xO})}\leq C||h_n||_{L^\infty(\partial\xO\cup K)}$ for some constant $C$ independent of $n$. By the previous case there exists a unique $L_\mu$-harmonic function $u_n$  satisfying
	\begin{align}
	\left|\frac{u_n(x)}{\tilde W(x)}-H_n(x)\right|\lesssim ||H_n||_{C^2(\overline{\xO})}d(x)d_K(x)^{\ap-\xg} \quad\forall x\in \xO\setminus K,\label{tofragma2}
	\end{align}
	where the implicit constant is independent of $n$. Thus $u_n\rightarrow u$ locally uniformly in $C^2(\xO\setminus K).$
	
	By \eqref{tofragma} and Lemma \ref{comparison}, we can easily show that
	$$\left| \frac{u_n(x)-u_m(x)}{\tilde W(x)}\right|\lesssim ||h_n-h_m||_{L^\infty(\partial\xO\cup K)} \quad \forall x \in \Omega \setminus K.$$
	Now, let $y\in \partial\xO\cup K.$ Then
	\begin{align*}
	\left|\frac{u(x)}{\tilde W(x)}-h(y)\right|
	\leq \left|\frac{u(x)-u_n(x)}{\tilde W(x)}\right|+\left|\frac{u_n(x)}{\tilde W(x)}-h_n(y)\right|+\left|h_n(y)-h(y)\right|.
	\end{align*}
	The result follows by letting successively $x\to y$ and $n\to\infty$.
\end{proof}


\section{Martin kernel}
In this section, several results can be obtained by using similar arguments as in \cite{GkV,caffa,hunt} with minor modifications, hence we will point out only precise references where the arguments can be found instead of providing detailled proofs. When the adaptation is not trivial, we offer detailled demonstration.

\subsection{$L_\mu$-harmonic measure}

Let $x_0\in\Omega \setminus K$ be a fixed reference point and $x \in \Omega \setminus K$. Let $\omega^{x_0}$ and $\omega^x$ the $L_\mu$-harmonic measures in $\partial \Omega \cup K$ relative to $x_0$ and $x$ respectively (the definition of $L_\mu$-harmonic measure is given after Definition \ref{kernel}). Thanks to the Harnack inequality, the measures $\xo^x$ and $\xo^{x_0}$, where $x_0,\;x\in \xO \setminus K$, are mutually absolutely continuous. For every fixed $x \in \Omega \setminus K$, we denote the Radon-Nikodyn derivative by $K_\mu(x,y)$ as in \eqref{RN-derivative}.

Let $\xi\in\partial\xO\cup K.$ We set $\xD_r(\xi)=(\partial\xO\cup K)\cap B_r(\xi)$ and
$x_r=x_r(\xi)\in\xO \setminus K$ such that $d(x_r)=|x_r-\xi|=r$ if $\xi\in\partial\xO$ or $d_K(x_r)=|x_r-\xi|=r$ if $\xi\in K.$ Also, if $\xi\in \partial\xO$ then $x_{r}(\xi)=\xi-r{\bf n}_{\xi}$ where $\bf n_{\xi}$ is the unit outward normal vector to $\prt\Gw$ at $\xi$. We recall that $\beta_0>0$ is the constant given in \eqref{straigh}.

\begin{lemma}\label{lem2.1}
	For any $0<r\leq \frac{\beta_0}{4}$ and $\xi\in\partial \Omega \cup K$, there holds
	\bel{00}
	\frac{\omega^x(\xD_r(\xi))}{\tilde W(x)}\gtrsim 1\qquad\forall x\in(\Omega \setminus K)\cap B_{\frac{r}{2}}(\xi).
	\ee
\end{lemma}
\begin{proof}
We consider only the case  $\mu=H^2$ and $\xi\in K$ since the proof in the other cases is very similar. Let $0\leq h\leq1$ be a smooth function with compact support in $\xD_r(\xi)$ such that $h=1$ on $\overline{\xD_\frac{3r}{4}}(\xi)$. Let $v_h$ be the unique solution of \eqref{linear} and $v_1$ is the solution with $h=1$. Then $v_1\geq v_h$  and
	$$\lim_{x\in\xO\setminus K,\;x\rightarrow x_0}\frac{v_1(x)-v_h(x)}{\tilde W(x)}=0\qquad\forall x_0\in(\Omega \setminus K)\cap B_{\frac{3r}{4}}(\xi).$$
	By Lemma \ref{lemharnack} and the fact that $\phi_{\mu }\approx d_K^{-\am}$, it follows
	$$d_K(x)^{\am}(v_1(x)-v_h(x))\lesssim d_K(y)^{\am}(v_1(y)-v_h(y))\qquad\forall x,y\in (\xO\setminus K)\cap \overline{B_{\frac{r}{2}}(\xi)}.$$
	
	By \eqref{ektimisi}, we have
	$$0\leq d_K^{H}(x)(v_1(x)-v_h(x))\lesssim d_K^{H}(y)(v_1(y)-v_h(y))\lesssim |\ln d_K(y) |,\forall x,y\in (\Omega \setminus K)\cap \overline{B_{\frac{r}{2}}(\xi)}.$$
	Thus, combining the above estimates, we have that
	$$\frac{d_K^{H}(x)v_1(x)}{|\ln d_K(x)|}-\frac{1}{c'}\frac{|\ln d_K(y) |}{|\ln d_K(x) |}\leq \frac{d_K^H(x) v_h(x)}{|\ln d_K(x)|}.$$
	Now in view of the proof of Lemma \ref{mainlemma1}, there exists $\xe_0>0$ such that $$\frac{d_K^{H}(x)v_1(x)}{|\ln d_K(x)|}>\frac{1}{2}\qquad\forall x\in K_{\xe_0} .$$
	Thus if we choose $y$ such that $d(y)=\frac{r}{4},$ there exists a constant $D_0=D_0(\beta_0,c',\varepsilon_0)>0$ such that
	$$\frac{1}{c'}\frac{|\ln d_K(y) |}{|\ln d_K(x) |}= \frac{1}{c'}\frac{|\ln \frac{r}{4} |}{|\log d_K(x) |}\leq \frac{1}{c'}\frac{|\ln \frac{r}{4} |}{|\ln \frac{r}{D_0} |}
	\leq \frac{1}{4}\qquad\forall x\in K_{\frac{r}{D_0}}$$
	and
	\bel{111}
	d_K^{H}(x)\frac{v_h(x)}{|\ln d_K(x)|}\geq\frac{1}{4}\qquad\forall x\in \overline{B_{\frac{r}{2}}}(\xi)\cap K_{\frac{r}{D_0}}.
	\ee
	In particular
	\bel{112}
	(a^*r)^H\frac{v_h(x_{a^*r}(\xi))}{|\ln (a^*r)|}\geq\frac{1}{4},
	\ee
	where $a^*=(\max\{2,D_0\})^{-1}$. If $D_0\leq 2$ we obtain the claim. If $D_0> 2$, set $k^*=\BBE[\frac{D_0}{2}]+1$ (we recall that $\BBE[x]$ denotes the largest integer less than or equal to $x$). If $x\in \overline{B_{\frac{r}{2}}}(\xi)\cap (K_{\frac{r}{D_0}})^c$ there exists a chain of at most $4k^*$ points $\{z_j\}_{j=0}^{j_0}$ such that $z_j\in\overline{B_{\frac{r}{2}}}(\xi)\cap\Gw$, $d_K(z_j)\geq a^*r$, $ z_0=x_{a^*r}(\xi)$, $z_{j_0}=x$ and  $|z_j-z_{j+1}|\leq \frac{a^*r}{4}$. By Harnack inequality (applied $j_0$ times)
	\be\label{113}
	v_h(x_{a^*r}(\xi))\leq c v_h(x).
	\ee
	Since $d_K(x_{a^*r}(\xi))=a^*r\leq d_K(x)$,
	we obtain finally
	\bel{114}
	\frac{1}{4}\leq (a^*r)^H\frac{v_h(x_{a^*r}(\xi))}{|\ln (a^*r)|}\leq c \frac{\gw^x(\Gd_r(\xi))}{\tilde W(x)}\qquad
	\forall x\in (\Gw\setminus K)\cap B_{\frac{r}{2}}(\xi).
	\ee
\end{proof}

\begin{lemma}\label{lemharn}
	For any $\xi\in\prt\Gw\cup K$ and $0<r\leq s\leq \frac{\beta_0}{4}$,
	\be\label{carle}
	\frac{\omega^x(\Delta_{r}(\xi))}{\tilde W(x)}\lesssim \frac{ \omega^{x_{s}(\xi)}(\Delta_{r}(\xi))}{\tilde W(x_{s}(\xi))}\qquad\forall x\in (\Omega \setminus K)\setminus B_{s}(\xi).
	\ee
\end{lemma}
\begin{proof}

	Let $h\in C(\partial\Omega \cup K)$ with compact support in $\xD_r(\xi),$ $h=1$ in $\overline{\Delta}_\frac{r}{2}(\xi)$ and $0\leq h\leq 1$. Let $v_h$ be the unique solution of \eqref{linear}. By Lemma \ref{mainlemma1},  for any  $0<r<\gb_0$,
	\bel{glob}
	\frac{v_h(x)}{\tilde W(x)}\leq\frac{\omega^x(\Delta_r(\xi))}{\tilde W(x)}\leq \frac{\omega^x(\prt\Gw \cup K)}{\tilde W(x)}\lesssim 1 \qquad\forall x\in \Omega \setminus K.
	\ee
	By Lemma \ref{mainlemma1} , there holds
	\begin{align}\label{glob00}
	\lim_{\min\{d(x),d_K(x) \} \to 0}\frac{v_1(x)}{\tilde W(x)}&=1.
	\end{align}
	Thus we can replace $\tilde W$ by $v_1$ (the unique solution of \eqref{linear} with $h \equiv 1$) in \eqref{carle}. Since $w_h=\frac{v_h(x)}{v_1(x)}$ is H\"older continuous in $\overline\Gw$ and  satisfies
	\bel{glob01} \left\{ \BAL
	-\div (v_1^2\nabla w_h)&=0\qquad &&\text{in }(\Gw\setminus K)\setminus \overline {B_{s}}(\xi)\\
	0\leq w_h &\leq 1 &&\text{in }(\Gw\setminus K)\setminus \overline {B_{s}}(\xi)\\
	w_h&=0\qquad \qquad &&\text{in }(\prt\Gw\cup K)\setminus \overline {B_{s}}(\xi),
	\EAL \right. \ee
	the maximum of $w_h$ is achieved on $(\Gw\setminus K) \cap\prt B_{s}(\xi)$, therefore it is sufficient to prove the Carleson estimate
	$$ w_h(x)\leq C w_h(x_{s}(\xi)) \quad  \forall x\in (\Omega \setminus K) \cap\prt B_{s}(\xi).$$
	If $x$ such that $|x-\xi|=s$ is ``far" from  $\prt\Gw\cup K$, $w_h(x)$ is ``controlled" by $w_h(x_{s}(\xi))$ thanks to Harnack inequality, while
	if it is close to $\prt\Gw\cup K$, $w_h(x)$ is ``controlled'' by the fact that it vanishes on $(\prt\xO\cup K) \cap\prt B_{s}(\xi)$. \smallskip

The rest of the proof is very similar to the proof of \cite[Lemma 2.20]{GkV} and we omit it.
	
\end{proof}
\begin{theorem}\label{lem2.2*} Assume $\mu \leq H^2$.
	For any $0<r\leq\frac{\beta_0}{4}$ and $\xi\in K$, the followings hold.
	
	(i) If $\mu=H^2$ and $\xi\in K$ then
	\bel{Rep0}\BAL
	\omega^x(\xD_r(\xi))\approx  r^{N-2-H}|\ln r|G_{H^2}(x_r(\xi),x) \quad \forall x\in(\Omega \setminus K)\setminus B_{4r}(\xi).
	\EAL
	\ee
	
	(ii) If $\xm<H^2$ and $\xi\in K$ then
	\be\label{Rep0b}\BAL
	\omega^x(\xD_r(\xi))\approx  r^{N-2-\ap}G_{\mu}(x_r(\xi),x) \quad \forall x\in(\Omega \setminus K)\setminus B_{4r}(\xi).
	\EAL
	\ee

    (iii)  If $\xm\leq H^2$ and $\xi\in \partial\xO$ then
    $$
    \xo^x(\xD_r(\xi))\approx r^{N-2} G_{\xm}(x_r(\xi),x) \quad \forall x\in(\xO\setminus K)\setminus B_{4r}(\xi).
    $$	
\end{theorem}  
\begin{proof}
The proof can be proceeded as in the proof of \cite[Theorem 2.22]{GkV} with minor modifications, hence we omit it.
\end{proof}\medskip

As a consequence of Theorem \ref{lem2.2*} and the Harnack inequality, $L_\mu$-harmonic measures possess the doubling property.
\begin{theorem}\label{Th10}
	Let $\xm \leq H^2$.  For any $0<r\leq \frac{\gb_0}{4}$, there holds
	$$\xo^x(\xD_{2r}(\xi))\lesssim \xo^x(\xD_{r}(\xi))\qquad\forall x\in(\xO\setminus K)\setminus B_{4r}(\xi).$$
\end{theorem}\smallskip

\begin{lemma}\label{Lemm10}
	Let $0<r\leq \gb_0$ and $\xi \in \partial \Omega \cup K$. Assume $u$ is a positive $L_{\mu }$-harmonic function in $\Omega \setminus K$ such that $\frac{u}{
	\tilde W}$ can be extended as a continuous function on $C(\overline{\Omega \setminus B_r(\xi)})$ and
	$$\lim_{x \in \Omega \setminus K, x \to z}\frac{u(x)}{\tilde W(x)}=0 \quad \text{uniformly with respect to } z \in(\partial\Omega \cup K)\setminus \overline{B_{r}(\xi)}. $$

	Then
	$$u(x)\approx \frac{u(x_r(\xi))}{\tilde W(x_r(\xi))}\omega^x(\xD_r(\xi))
	\qquad\forall x\in(\Omega \setminus K)\setminus\overline{B_{2r}(\xi)}. $$
\end{lemma}  
\begin{proof}
	It follows from Lemma \ref{lemharnack} that
	$$ \frac{u(x)}{\omega^{x}(\xD_r(\xi))} \approx \frac{u(x_{2r}(\xi))}{\omega^{x_{2r}(\xi)}(\xD_r(\xi))}\qquad\forall x\in(\xO\setminus K) \cap \partial B_{2r}(\xi).$$
	Applying Harnack inequality between $x_{2r}(\xi)$ and $x_{r}(\xi)$ we obtain
	$$ \frac{u(x)}{\omega^{x}(\xD_r(\xi))}\approx \frac{u(x_{r}(\xi))}{\omega^{x_{r}(\xi)}(\xD_r(\xi))} \qquad\forall x\in (\xO\setminus K)\cap \partial B_{2r}(\xi).$$
	Also by Harnack inequality we have that
	$\omega^{x_{r}(\xi)}(\xD_r(\xi))\gtrsim  \omega^{x_{\frac{r}{2}}(\xi)}(\xD_r(\xi)) \gtrsim \tilde W(x_r(\xi))$
	where in the last inequality above we have used Lemma \ref{lem2.1}.
	
	On the other hand, from \eqref{glob}, we have $\omega^{x_{r}(\xi)}(\xD_r(\xi))\lesssim \tilde W(x_r(\xi))$.
	Combining the above inequalities,  we derive
	$$u(x)\approx \frac{u(x_r(\xi))}{\tilde W(x_r(\xi))}\omega^x(\xD_r(\xi))\qquad\forall x\in (\xO\setminus K)\cap \partial B_{2r}(\xi).$$
	The result follows by an argument similar to step 3 in Lemma \ref{lemharn}.
\end{proof}

\subsection{Martin kernel of $-L_{\mu}$} We first give the existence and uniqueness of the kernel function of $-L_\mu$ which is defined in Definition \ref{kernel}.
\begin{proposition} \label{uniq}
	There exists one and only one kernel function for $-L_{\xm }$  with pole at $\xi$ and with basis at $x_0$.
\end{proposition}
\begin{proof}
	The proof is similar to that of \cite[Theorem 3.1]{caffa} and hence we omit it.
\end{proof}

In view of the proof of Proposition \ref{uniq} (in fact from \cite[Theorem 3.1]{caffa}) and by the uniqueness, the function $K_\mu$ defined in \eqref{RN-derivative} is the unique kernel function of $-L_\mu$ and
$$ K_\mu(x,\xi)= \lim_{r \to 0}\frac{\omega^x(\Delta_r(\xi))}{\omega^{x_0}(\Delta_r(\xi))} \quad \text{for } \omega^{x_0}-\text{a.e. } \xi \in \partial \Omega \cup K.
$$

\begin{proposition} \label{continuous}
	For any $x\in\Gw \setminus K$, the function $\xi\mapsto K_{\mu}(x,\xi)$ is continuous on $\prt\xO\cup K$.
\end{proposition}
\begin{proof}
	The proof is similar to the one of \cite[Corollary 3.2]{caffa} and hence we omit it.
\end{proof}
We can now identify the Martin boundary and topology with their classical
analogues. We begin by recalling the definitions of the Martin boundary and
related concepts.

For $x,\;y\in \Omega \setminus K$, we set
$$\mathcal{K}_\xm(x,y):=\frac{G_{\mu}(x,y)}{G_{\mu}(x_0,y)}.$$

Consider the family
of sequences $\{y_k\}_{k\geq1}$ of points of $\xO\setminus K$ without cluster points in $\xO\setminus K$ for which
$\mathcal{K}_\xm(x,y_k)$ converges in $\xO\setminus K$ to a harmonic function, denoted by $\mathcal{K}_\xm(x,\{y_k\})$. Two
such sequences $\{y_k\}$ and $\{y_k'\}$ are called equivalent if $\mathcal{K}_\xm(x,\{y_k\})=\mathcal{K}_\xm(x,\{y_k'\})$
and each equivalence class is called an element of the Martin boundary $\xG.$ If $Y$ is
such an equivalence class (i.e., $Y\in \xG$) then $\mathcal{K}_\xm(x,Y)$ will denote the corresponding
harmonic limit function. Thus each $Y\in(\xO\setminus K)\cup\xG$ is associated with a unique
function $\mathcal{K}_\xm(x,Y).$ The Martin topology on $(\xO\setminus K)\cup\xG$ is given by the metric
$$\xr(Y,Y')=\int_{A}\frac{|\mathcal{K}_\xm(x,Y)-\mathcal{K}_\xm(x,Y')|}{1+|\mathcal{K}_\xm(x,Y)-\mathcal{K}_\xm(x,Y')|}dx\quad Y,Y'\in (\xO\setminus K)\cup \xG,$$
where $A$ is a small enough neighborhood of $x_0.$ $\mathcal{K}_\xm(x,Y)$ is a $\xr-continuous$
function of $Y\in(\xO\setminus K)\cup\xG$  for $x \in \xO\setminus K$ fixed, $(\xO\setminus K)\cup\xG$ is compact and complete with respect
to $\xr$, $(\xO\setminus K)\cup\xG$ is the $\xr$-closure of $\xO\setminus K$ and the $\xr$-topology is equivalent to the Euclidean
topology in $\xO\setminus K$. We have the following results.

\begin{proposition} \label{martincorresp}
	There is a one-to-one correspondence between the Martin boundary of $\Omega \setminus K$ and the Euclidean boundary $\partial (\Omega \setminus K) = \partial \Omega \cup K$. If $Y \in\xG$ corresponds
	to $\xi\in \partial \Omega \cup K$ then $\mathcal{K}_\xm(x,Y)=K_{\mu}(x,\xi).$ The Martin topology on $(\Omega \setminus K) \cup \Gamma$ is equivalent
	to the Euclidean topology on $(\Omega \setminus K) \cup (\partial\xO \cup K).$
\end{proposition}
\begin{proof}
	The proof is similar to the one of \cite[Theorem 4.2]{hunt} and we omit it.
\end{proof}

\begin{proof}[\textbf{Proof of Theorem \ref{Martin}}.] From Proposition \ref{martincorresp}, we see that \eqref{martindef} holds, it means $K_\mu$ is the Martin kernel of $-L_\mu$ in $\Omega \setminus K$. The continuity of $K_\mu$ is obtained in Proposition \ref{continuous}, while two-sided estimates \eqref{Martinest1} and \eqref{Martinest2} follow from \eqref{martindef} and estimates \eqref{Greenesta} and \eqref{Greenestb}.
\end{proof}

\smallskip

Let us now prove some $L^p$ estimates for the Martin kernel.

\begin{proposition} \label{pcritical1} Assume $\mu \leq H^2$ and $p>1$.
	
(i) If $y \in \partial \Omega$ then
\bel{ybdw} K_\mu(\cdot,y) \in L^p(\Omega;\phi_\mu) \Longleftrightarrow p<\frac{N+1}{N-1}.
\ee
Moreover, if $p<\frac{N+1}{N-1}$ then
\bel{ybdwb} \int_{\Omega} K_\mu(x,y)^p\phi_\mu(x)dx  \approx 1.
\ee
Here the similarity constants in \eqref{ybdwb} depend only on $N,\Omega,K,\mu,p$.

(ii) If $y \in K$ then
\bel{yK} K_\mu(\cdot,y) \in L^p(\Omega;\phi_\mu) \Longleftrightarrow p<\frac{N-\am}{N-2-\am}.
\ee
Moreover, if $p<\frac{N-\am}{N-2-\am}$, then  \eqref{ybdwb} holds.
\end{proposition}
\begin{proof}
We prove only (ii) since (i) can be obtained by a similar argument. Assume $y \in K$.

\noindent \textbf{Case 1: $0<\mu < \left( \frac{N-2}{2} \right)^2$.}  From \eqref{eigenfunctionestimates} and \eqref{Martinest1}, we obtain
\bel{yK1}	
\phi_\mu(x) K_\mu(x,y)^p \approx d(x)^{p+1}d_K(x)^{-\am(p+1)}|x-y|^{-p(N-2-2\am)}.
\ee

If $p<\frac{N-\am}{N-2-\am}$ then by Lemma \ref{anisotita}, we obtain
\bel{yK2} \int_{\Omega}\phi_\mu(x) K_\mu(x,y)^p dx \lesssim \int_{\Omega}d(x)^{p+1}d_K(x)^{-\am(p+1)}|x-y|^{-p(N-2-2\am)}dx \lesssim 1.
\ee
Therefore
\bel{impli1-K}
p<\frac{N-\am}{N-2-\am} \Longrightarrow K_\mu(\cdot,y) \in L^p(\Omega;\phi_\mu).
\ee

On the other hand, for $x \in K_{\beta_0}$, we have $d(x) \gtrsim 1$ and $d_K(x) \leq |x-y|$. Therefore,
\begin{equation} \label{yK2'} \begin{aligned}
\int_{\Omega}\phi_\mu(x) K_\mu(x,y)^p dx &\gtrsim \int_{K_{\beta_2}}  |x-y|^{- \am(p+1) - p(N-2-2\am)}dx.
\end{aligned} \end{equation}
This implies
\bel{yK3}  K_\mu(\cdot,y) \in L^p(\Omega;\phi_\mu) \Longrightarrow p<\frac{N-\am}{N-2-\am}. \ee
Combining \eqref{yK2} and \eqref{yK3} yields \eqref{yK}. Moreover, we derive \eqref{ybdwb} from \eqref{yK2} and \eqref{yK2'}.

\noindent \textbf{Case 2: $\mu \leq 0$.} If $p<\frac{N-\am}{N-2-\am}$, since $d_K(x)^{-\am(p+1)} \leq |x-y|^{-\am(p+1)}$, it follows that
\bel{yK4} \int_{\Omega}\phi_\mu(x) K_\mu(x,y)^p dx \lesssim \int_{\Omega}d(x)^{p+1}d_K(x)^{-\am(p+1)}|x-y|^{-p(N-2-2\am)}dx \lesssim 1.
\ee
Consequently, \eqref{impli1-K} holds.

On the other hand, for $x \in V(y,\frac{\beta_1}{2})$, estimate \eqref{propdist} holds. This leads to
\begin{equation} \label{yK4'} \begin{aligned}
\int_{\Omega}\phi_\mu(x) K_\mu(x,y)^p dx &\gtrsim \int_{V(y,\frac{\beta_1}{2})} d_K(x)^{- \am(p+1) - p(N-2-2\am)}dx \\
& \gtrsim \int_{B^{N-k}(0,\frac{\beta_1}{8})}|z''|^{- \am(p+1) - p(N-2-2\am)} dz''.
\end{aligned} \end{equation}
Therefore \eqref{yK3} holds and we obtain \eqref{yK}. Moreover, \eqref{ybdwb} follows from \eqref{yK4} and \eqref{yK4'}.

\noindent \textbf{Case 3: $k=0$, $K=\{0\}$ and $\mu=\left( \frac{N-2}{2} \right)^2$.} By proceeding as above, using \eqref{eigenfunctionestimates} and \eqref{Martinest2}, we derive \eqref{yK}and \eqref{ybdwb}.
\end{proof}

\begin{corollary} \label{cor:pcritical} Assume $\mu \leq H^2$.

(i) If $1 < p <\frac{N+1}{N-1}$ then for any $\nu \in \GTM(\partial \Omega \cup K)$ with compact support on $\partial \Omega$, there holds
\bel{Kpnu1}
\|  \BBK_\mu[\nu] \|_{L^p(\Omega;\ei)} \lesssim \| \nu \|_{\GTM(\partial \Omega)}.
\ee

(ii) If $1< p <\frac{N-\am}{N-2-\am}$ then for any $\nu \in \GTM(\partial \Omega \cup K)$ with compact support on $K$, there holds
\bel{Kpnu2}
\|  \BBK_\mu[\nu] \|_{L^p(\Omega;\ei)} \lesssim \| \nu \|_{\GTM(K)}.
\ee

(iii) If $1< p <\min\left\{ \frac{N+1}{N-1},\frac{N-\am}{N-2-\am} \right\}$ then for any $\nu \in \GTM(\partial \Omega \cup K)$, there holds
\bel{Kpnu3}
\|  \BBK_\mu[\nu] \|_{L^p(\Omega;\ei)} \lesssim \| \nu \|_{\GTM(\partial \Omega \cup K)}.
\ee
The implicit constants depend on $N,\Omega,K,\mu,p$.
\end{corollary}
\begin{proof}
By using Proposition \ref{pcritical1} and Jensen's inequality, we obtain easily (i)--(iii).
\end{proof}

\subsection{Representation theorem}

Let us give a lemma that we will use to prove the representation formula.
\begin{lemma} \label{Vsuperhar}
	Let $F \subset\partial\xO\cup K$ and $D$ be an open smooth neighborhood of $F$. Assume that if $F\subset K$ then  $D\Subset \xO$ and if $F\subset \partial\xO$ then $\xO\cap D\subset (\xO \setminus K_\xb)$ for some $\xb>0.$ Let $u$ be a positive $L_\xm$-harmonic function in $\xO\setminus K.$ Then there exists  $L_\mu$-superharmonic function $V$ such that
	$$
	V(x)=\left\{\BA {lll}v(x)\qquad&\text{in }\;(\xO\setminus K)\setminus D\\[2mm]
	u(x)\qquad&\text{in }\;(\xO\setminus K)\cap \overline{D},
	\EA\right.
	$$
	where $v$ satisfies
	\bel{tyx} \left\{ \BAL L_\mu v&=0\qquad &&\text{in }(\xO\setminus K)\setminus \overline{D}\\
	\lim_{x \in \Omega \setminus K,\; x \to y}v(x)&=u(y) &&\forall y\in \partial D\cap(\xO\setminus K)\\
	\lim_{x \in \Omega \setminus K,\; x\to y}\frac{v(x)}{\tilde W(x)}&=0 &&\forall y\in(\partial\Omega \cup K) \setminus \overline{D}.
	\EAL \right. \ee
\end{lemma}
\begin{proof}
	Let $F\subset K$. Note that  $u \in C^2(\Omega \setminus K)$ since it is $L_\mu$-harmonic. We assume that $\{r_n\}_{n=0}^\infty$ is a decreasing sequence such that $r_n\searrow0$ and $r_1<\frac{\xb_0}{16}.$ Set
	$$D_{r_n}:=\{\xi \in  \partial D\cap\Omega : d(\xi)>2r_n\}.$$
	
	Let $0\leq\eta_n\leq1$ be a smooth function such that $\eta_n=1$ in  $\overline{D}_{r_n}$ with compact support in $D_{\frac{r_n}{2}}$. In view of the proof of Lemmata \ref{mainlemma1} and \ref{existence1}, for $m>n$, there is a unique solution $v_{n,m}$ of
	\be \left\{ \BAL L_\mu v &=0\qquad &&\text{in }(\Omega \setminus K_{\frac{r_m}{2}})\setminus \overline{D}\\
	\lim_{x \to y}v(x)&=\eta_n(y)u(y)&&\forall y\in \partial D\cap(\Omega \setminus K_{\frac{r_m}{2}})\\
	\lim_{x\to y}v(x)&=0 &&\forall y\in(\partial\xO\cup \partial K_{\frac{r_m}{2}})\setminus \overline{D}.
	\EAL \right. \ee
	
	Furthermore, by the comparison principle, we have $0\leq v_{n,m} \leq u$ for  $m>n$ and hence $v_{n,m} \leq v_{n,m+1}$. 	In addition, there exists a constant $c_n=c_n(||u||_{L^\infty(D_\frac{r_n}{2})}, \inf_{x\in D_\frac{r_n}{2}}\xf_\xm)$ such that
	$0\leq v_{n,m}(x)\leq\min\{u(x),c_n\xf_\xm(x)\}$ for all $x\in \xO\setminus (D\cup K_{r_m})$ and $n,m\in \mathbb{N}$.
	Thus $v_{n,m}\to v_n$ locally uniformly in $\xO\setminus (\overline{D}\cup K)$ as $m\to \infty$ and hence
	\bel{fragma12}
	0\leq v_{n}(x)\leq\min\{u(x),c_n\xf_\xm(x)\} \quad  \forall x\in \xO\setminus (\overline{D}\cup K)\;\;\text{and}\;\;\forall n\in \mathbb{N}.
	\ee
	For all $\xi\in K\setminus \overline{D},$ by \eqref{fragma12} and \eqref{harnack} there exists $r_0<\frac{\dist(\xi,\partial D)}{4}$ such that
	$$\frac{v_n(x)}{\phi_\mu(x)}\lesssim \frac{v_n(y)}{\phi_\mu(y)}\lesssim \frac{u(y)}{\phi_\mu(y)}, \quad\forall x,y \in B_{\frac{r_0}{4}}(\xi) \cap (\Omega \setminus K).$$
	Thus $v_n\to v$ and the desired result follows if we set $V=u \land v$.
	
	If $\xi\in \partial\xO$ the proof is similar and simpler and thus we omit it.
\end{proof} \medskip

We recall that $x_0 \in \Omega \setminus K$  is a fixed reference point. Let $\{\xO_n\}$ be an increasing sequence of bounded open smooth domains  such that
\bel{Omegan}  \overline{\xO_n}\subset \xO_{n+1}, \quad \cup_n\xO_n=\xO, \quad \mathcal{H}^{N-1}(\partial \Omega_n)\to \mathcal{H}^{N-1}(\partial \Omega).
\ee
  Let $\{K_n\}$ be a decreasing sequence of bounded open smooth domains  such that
\bel{Kn} K\subset K_{n+1}\subset\overline{K_{n+1}}\subset K_{n}\subset\overline{K_{n}} \subset\Omega_n, \quad \cap_n K_n=K.
\ee
Set $O_n=\xO_n\setminus K_n$ for each $n$ and assume that $x_0 \in O_1$. Such a sequence $\{O_n\}$ will be called a {\it smooth exhaustion} of $\Gw\setminus K$.

Then $-L_\mu$ is uniformly elliptic and coercive in $H^1_0(O_n)$ and its first eigenvalue $\lambda_\mu^{O_n}$ in $O_n$ is larger than its first eigenvalue $\lambda_\mu$ in $\Omega \setminus K$.

For $h\in C(\prt O_n)$, the following problem
\be\label{sub} \left\{ \BAL
-L_{\xm } v&=0\qquad&&\text{in } O_n\\
v&=h\qquad&&\text{on } \prt O_n,
\EAL \right.
\ee
admits a unique solution which allows to define the $L_{\xm }$-harmonic measure $\omega_{O_n}^{x_0}$ on $\prt O_n$
by
\be\label{redu2}
v(x_0)=\myint{\prt O_n}{}h(y)d\gw^{x_0}_{O_n}(y).
\ee

\begin{proposition}\label{traceW}
	For every $\phi \in C(\overline{\xO})$, there holds
	\be\label{2.27}
	\lim_{n\rightarrow\infty}\int_{\partial O_n} \phi(x)\tilde W(x)d\gw^{x_0}_{O_n}(x)=\int_{\partial \xO\cup K}\phi(x)d\gw^{x_0}(x).
	\ee
\end{proposition}  
\begin{proof}
	Let $n_0\in \mathbb{N}$ be such that
	$\hbox{dist}(\partial K_n,K)<\frac{\xb_0}{16}$ for all $n \geq n_0$.
	For $n\geq n_0$, let $w_n$ be the solution of
	\be\label{sub-n} \left\{  \BAL
	-L_{\mu }w_n&=0 \qquad&&\text{in } O_n\\
	w_n&=\tilde W\qquad&&\text{on } \prt O_n.
	\EAL \right.
	\ee
	In view of the proof of Lemma \ref{mainlemma1}, there exists a positive constant $c=c(\Omega,K,\mu)$ such that
	$$\left \|\frac{w_n}{\tilde W} \right \|_{L^\infty(O_{n_0})}\leq c \quad \forall n\geq n_0.$$
	Furthermore
	\be
	w_n(x_0)=\int_{\partial O_n}\tilde W(x)d\gw^{x_0}_{O_n}(x)<c.\label{wn}
	\ee
	We extend $\gw^{x_0}_{O_n}$ as a Borel measure on $\overline{\xO}$ by setting $\gw^{x_0}_{O_n}(\overline{\xO}\setminus O_n) = 0,$ and keep
	the notation $\gw^{x_0}_{O_n}$ for the extension. By \eqref{wn}, the sequence $\{\tilde W\gw^{x_0}_{O_n}\}$ is bounded in the space $\mathfrak M_b(\overline\Gw)$ of bounded Borel measures in $\overline\Gw$. Thus there exists
	a subsequence, still denoted by $\{\tilde W\gw^{x_0}_{\Gw_n}\}$, which converges narrowly to some
	positive measure, say $\widetilde{\gw}$ which is clearly supported on $\partial\xO\cup K$ and satisfies
	$\|\widetilde{\gw}\|_{\mathfrak M_b(\partial \Omega \cup K)}\leq c$ due to \eqref{wn}. For any
	$\phi \in C(\overline{\xO})$ there holds
	$$\lim_{n\rightarrow\infty}\int_{\partial O_n}\phi(x)\tilde W(x)d \gw^{x_0}_{O_n}(x)=\int_{\partial\xO\cup K}\phi(x) d \widetilde{\gw}(x).$$
	Set $\xz:=\phi\lfloor_{\prt\Gw\cup K}$ and $z(x):=\int_{\partial\xO\cup K}K_{\mu}(x,y)\xz(y)d \gw^{x_0}(y)$.
	Then
	$$\lim_{\Omega \setminus K \ni x \to y}\frac{z(x)}{\tilde W(x)}=\gz(y) \quad \forall y \in \partial \Omega \cup K  \quad \text{ and }\quad z(x_0)=\int_{\partial\xO\cup K}\gz(y) d \gw^{x_0}(y).
	$$
	
	By Lemma \ref{mainlemma1}, $\frac{z}{\tilde W}\in C(\overline{\xO})$. Since $\frac{z}{\tilde W}\lfloor_{\prt O_n} \to \zeta$ uniformly as $n\to\infty$, there holds
	$$z(x_0)=\int_{\partial O_n}z\lfloor_{\prt O_n}d \gw^{x_0}_{O_n}=\int_{\partial O_n}\tilde W\frac{z\lfloor_{\prt O_n}}{\tilde W}d \gw^{x_0}_{O_n}\to \int_{\partial\xO\cup K}\gz d \tilde\gw\;\text{ as }\;n\to\infty.$$
	It follows that
	$$\int_{\partial\xO\cup K}\xz d \widetilde{\gw}=\int_{\partial\xO\cup K}\xz d \gw^{x_0}\qquad\forall\xz\in C(\partial\xO \cup K). $$
	Consequently $d\widetilde{\gw}=d \gw^{x_0}.$ Because the limit does not depend on the
	subsequence it follows that the whole sequence $\{\tilde Wd\gw^{x_0}_{O_n}\}$ converges weakly to $\omega^{x_0}$. This implies \eqref{2.27}.
\end{proof} \smallskip

\begin{proof}[\textbf{Proof of Theorem \ref{th:Rep}}.]

(i) It can be seen that, for any $\nu \in \GTM^+(\partial \Omega \cup K)$, $\BBK_\mu[\nu]$ is an $L_\mu$-harmonic function in $\Omega \setminus K$.

Conversely, let $u$ be a positive $L_\mu$-harmonic function in $\Omega \setminus K$. We will show that there exists a unique measure $\nu \in \GTM^+(\partial \cup K)$ such that $u=\BBK_\mu[\nu]$. To do that, we will adapt the ideas in \cite[Theorem 4.3]{hunt}. Let $B$ be a relatively closed subset of $\Omega \setminus K$, we define
$$R^B_u(x):=\inf\{\psi(x):\;\psi\;\text{is nonnegative supersolution in}\; \xO\setminus K\; \text{with}\;\psi\geq u\;\text{on}\;B\}.$$
For a closed subset $F$ of $\partial\xO\cup K,$ we define
$$\xn^x(F):=\inf\{R^{(\xO\setminus K)\cap\overline{G}}_u(x):\;F\subset G,\;G\;\text{open in}\;\mathbb{R}^N \}.$$

The set function $\xn^x$ defines a regular Borel
measure on $\partial\xO\cup K$ for each fixed $x\in\xO\setminus K.$ Since $\xn^x(F)$ is a positive $L_\xm$-harmonic function
in $\xO\setminus K$, the measures $\xn^x$ are absolutely continuous with respect to $\xn^{x_0}$ by Harnack's
inequality. Hence,
$$\xn^x(F)=\int_Fd\xn^x(F)(y)=\int_F\frac{d\xn^x(F)}{d\xn^{x_0}(F)}d\xn^{x_0}(y).$$

We assert that
\begin{equation} \label{dv} \frac{d\xn^x(F)}{d\xn^{x_0}(F)} = K_{\mu}(x, y) \quad \text{for $\nu^{x_0}$-a.e. $y$ in $ \partial\xO \cup K$.}
\end{equation}
Indeed, by Besicovitch's
theorem,
$$\frac{d\xn^x(F)}{d\xn^{x_0}(F)}=\lim_{r\to0}\frac{\xn^x(\xD_{r}(y))}{\xn^{x_0}(\xD_r(y))},$$
for $\nu^{x_0}$-a.e. $y$ in $ \partial\xO \cup K$. By Lemma \ref{Vsuperhar} and in view of the proof of Proposition \ref{uniq} we have that $\frac{d\xn^x(F)}{d\xn^{x_0}(F)}$ is a kernel function, and by uniqueness of the Martin kernel,  the assertion \eqref{dv} follows. Hence $\xn^x(A)=\int_AK_{\mu}(x,y)d\xn^{x_0}(y)$ for all Borel $A\subset\partial\xO\cup K$ and in particular
$$u(x)=\xn^x(\partial\xO\cup K)=\int_{\partial\xO\cup K}K_{\mu}(x,y)d\xn^{x_0}(y).$$

Suppose now $u(x)=\int_{\partial\xO\cup K}K_{\mu}(x,y)d\xn(y)$ for a Borel measure $\xn$ on $\partial\xO \cup K$. For a closed set $F\subset  K$ we will show that $\xn(F)=\xn^{x_0}(F)$.

Choose a sequence of open sets $\{G_n \}$ in $\mathbb{R}^N$ such that $\cap_{n=1}^\infty G_n=F$ and
$$\xn^{x}(F)=\lim_{n \rightarrow\infty}R^{(\Omega \setminus K) \cap \overline{G}_n}_u(x).$$
Since $R^B_u(x)\leq R^A_u(x)$ if $B\subset A$, we can choose $G_n$ such that $\overline{G}_{n+1}\subset G_n,\;\forall n \geq1$ and $\Omega \setminus G_n$ to be a $C^2$ domain for all $n \geq1$. In view of the proof of Lemma \ref{Vsuperhar}, we may assume that $R^{(\Omega \setminus K) \cap \overline{G}_n}_u(x)=V_n$ where $V_n$ is the $L_\mu$-superharmonic in Lemma \ref{Vsuperhar} for $D=G_n$. Furthermore we have that  $R^{(\Omega \setminus K) \cap \overline{G}_n}_u(x)=u(x)$ in $\overline{G}_n \cap (\Omega \setminus K)$ and $R^{(\Omega \setminus K) \cap \overline{G}_n}_u(x) \leq u(x)$ for all $x\in \Omega \setminus K$.

Let $\{ \Omega_n \}$ and $\{K_n\}$  be sequences satifying \eqref{Omegan} and \eqref{Kn} respectively. We may assume that $ G_n \Subset K_n$ for any $n \in \N$. Set $O_n=\xO_n\setminus K_n$ and denote by $\omega^{x_0}_{O_n}$  the $L_\xm$-harmonic measure in $\partial O_n$ (see \eqref{sub}-\eqref{redu2}).
Let $n>\ell$  and $v_n$ be the unique solution of $L_\xm v=0$ in $O_n$ and $v_n=R^{(\Omega \setminus K) \cap \overline{G}_\ell}_u(x)$ on $\partial O_n$. Since $R^{(\Omega \setminus K) \cap \overline{G}_\ell}_u(x)$ is supersolution in $\xO\setminus K,$ we have that $R^{(\Omega \setminus K) \cap \overline{G}_\ell}_u(x)\geq v_n(x)$ for any $x\in O_n$. Hence
\begin{align*}
R^{(\Omega \setminus K) \cap \overline{G}_\ell}_u(x_0)\geq v_n(x_0) =\int_{\partial{O_n}}R^{(\Omega \setminus K) \cap \overline{G}_\ell}_u(y)d\omega^{x_0}_{O_n}(y)
\geq\int_{\partial  O_n\cap G_\ell}R^{(\Omega \setminus K) \cap \overline{G}_\ell}_u(y)d\omega^{x_0}_{O_n}(y).
\end{align*}
Now, by Lemma \ref{Vsuperhar},
\begin{align*}
\int_{\partial O_n\cap G_\ell}R^{(\Omega \setminus K) \cap \overline{G}_\ell}_u(y)d\omega^{x_0}_{O_n}(y)
&=\int_{\partial O_n\cap G_\ell}u(y)d\omega^{x_0}_{O_n}(y)\\
&\geq\int_{F_m}\int_{\partial O_n\cap G_\ell}K_{\mu}(y,\xi)d\omega^{x_0}_{O_n}(y)d\xn(\xi),
\end{align*}
where $F_m\subset F,$ $\cup F_m=F$ and $\dist(F_m, K\setminus F)>\frac{1}{m}.$
If $\xi\in F_m$ we have
\begin{align*}
1=K_{\mu}(x_0,\xi)&=\int_{\partial O_n\cap G_\ell}K_{\mu}(y,\xi)d\omega^{x_0}_{O_n}(y)+\int_{\partial{O_n}\setminus G_\ell}K_{\mu}(y,\xi)d\omega^{x_0}_{O_n}(y).
\end{align*}
But $K_{\mu}(y,\xi)\leq c(n) d^{-\am}_K(y)$ for all $y\in\partial{O_n}\setminus G_\ell$, thus by Proposition \ref{traceW} we have that
$$\lim_{n\rightarrow\infty}\int_{\partial{O_n}\setminus G_\ell}K_{\mu}(y,\xi)d\omega^{x_0}_{O_n}(y)=0.$$
Combining the above inequalities and Lebesgue's dominated convergence theorem yields
$$R^{(\Omega \setminus K) \cap \overline{G}_\ell}_u(x_0)\geq\lim_{n\rightarrow\infty}\int_{F_m}\int_{\partial O_n \cap G_\ell}K_{\mu}(y,\xi)d\omega^{x_0}_{O_n}(y)d\xn(\xi)=\xn(F_m).$$
Hence, letting $\ell \to\infty$ and $m\to \infty$ successively, we obtain $\nu^{x_0}(F) \geq \nu(F)$.

 For the reverse inequality, let $m>\ell$. Then
 \begin{align*}
R^{(\Omega \setminus K) \cap \overline{G}_\ell}_u(x_0)&=\int_{\partial O_\ell}R^{(\Omega \setminus K) \cap \overline{G}_\ell}_u(y)d\omega^{x_0}_{O_\ell}(y)\\
&=\int_{\partial O_\ell \cap G_{m}}R^{(\Omega \setminus K) \cap \overline{G}_\ell}_u(y)d\omega^{x_0}_{O_\ell}(y)+\int_{\partial O_\ell\setminus  G_{m}}R^{(\Omega \setminus K) \cap \overline{G}_\ell}_u(y)d\omega^{x_0}_{O_\ell}(y).
\end{align*}
 In view of the proof of Lemma \ref{Vsuperhar}, we deduce $R^{(\Omega \setminus K) \cap \overline{G}_\ell}_u(x)\leq Cd_K^{-\am}(x)$ for all $ x\in \xO\setminus G_m$.
 Thus by Proposition \ref{traceW} we have
$$\lim_{m\rightarrow\infty}\int_{\partial O_\ell \setminus  G_{m}}R^{(\Omega \setminus K) \cap \overline{G}_\ell}_u(y)d\omega^{x_0}_{O_\ell}(y)=0,$$
\begin{align*}
\int_{\partial O_\ell \cap G_{m}}R^{(\Omega \setminus K) \cap \overline{G}_\ell}_u(y)d\omega^{x_0}_{O_\ell}(y)&\leq\int_{\partial O_\ell\cap G_{m}}u(y)d\omega^{x_0}_{O_\ell}(y)\\
&=\int_{\partial\xO\cup K}\int_{\partial O_\ell \cap G_{m}}K_{\mu}(y,\xi)d\omega^{x_0}_{O_\ell}(y)d\xn(\xi).
\end{align*}
If $\xi\in(\partial\xO\cup K)\setminus G_{m}$, we infer again from Proposition \ref{traceW} that
$$\lim_{\ell \rightarrow\infty}\int_{\partial O_\ell \cap G_{m}}K_{\mu}(y,\xi)d\omega^{x_0}_{O_\ell}(y)=0.$$
If $\xi\in\partial\xO\cap G_{m}$, then
$$
\int_{\partial O_\ell \cap G_{m}}K_{\mu}(y,\xi)d\omega^{x_0}_{O_\ell}(y)\leq K_{\mu}(x_0,\xi)=1.
$$
Combining all the above inequalities, we obtain
$$\xn^{x_0}(F)=\lim_{\ell \rightarrow\infty}R^{(\Omega \setminus K) \cap \overline{G}_\ell}_u(x_0)\leq \int_{(\partial\xO\cup K)\cap\overline{G}_m}K_{\mu}(x_0,\xi)d\xn(\xi)=\xn((\partial\xO\cup K)\cap\overline{G}_m),$$
which implies $\nu^{x_0}(F)\leq \nu(F)$.
Thus we get the desired result in case $F \subset K$.

If $F\subset \partial \xO$ the proof is very similar and we omit it.
\end{proof}

\section{Boundary value problem for linear equations}

\subsection{Boundary trace}

We first examine the boundary trace of $\BBK_\mu[\nu]$.
\begin{lemma} \label{tracemartin} For any $\nu \in \GTM(\partial \Omega \cup K)$, $\tr(\BBK_\mu[\nu])=\nu$.	
\end{lemma}
\begin{proof}
	The proof is the same as the proof of Lemma 2.2 in \cite{MV} and we omit it.
\end{proof}
\begin{lemma} \label{tracegreen}
	Assume $\tau\in\mathfrak{M}(\xO\setminus K;\ei)$  and put $u=\mathbb{G}_{\mu}[\tau].$ Then $u\in W^{1,p}_{loc}(\xO\setminus K)$ for every $1<p<\frac{N}{N-1}$
	and $\tr(u)=0$.
\end{lemma}
\begin{proof}
	By \cite[Theorem 1.2.2]{MVbook}, $u\in W^{1,p}_{loc}(\xO\setminus K)$ for every $1<p<\frac{N}{N-1}$. Let $\{O_n\}$ be a smooth exhaustion of $\xO\setminus K$ (see \eqref{Omegan} and \eqref{Kn}) and $v_n$ be the unique solution of
	$$ \left\{ \BAL
	L_{\xm }^{O_n}v&=0\qquad&&\text{in } O_n\\
	v&=u\qquad&&\text{on } \prt O_n.
	\EAL \right.
	$$
	We note here that $v_n(x_0)=\int_{\prt O_n} u(y) d\gw^{x_0}_{O_n}(y)$.
	We first assume that $\tau\geq0$.  Let $G^{O_n}_{\mu}$ be the Green kernel of $-L_\xm$ in $O_n$, then $G^{O_n}_{\xm}(x,y)\nearrow G_{\mu}(x,y)$ for any $x,y\in \xO\setminus K$ and $x\neq y$. Put $\tau_n=\tau|_{O_n}$ and $u_n=\BBG_\mu^{O_n}[\tau_n]$ then $u_n\nearrow u$ a.e. in $\xO \setminus K$. By uniqueness we have that $u=u_n+ v_n$ a.e. in $O_n$. In particular, $u(x_0)=u_n(x_0)+v_n(x_0)$. This implies, by sending $n$ at infinity, that $\lim_{n \to \infty}v_n(x_0)=0$. Consequently, $\tr(u)=0$.
	
	In the general case, the result follows by the linearity.
\end{proof}

\begin{proof}[\textbf{Proof of Proposition \ref{traceKG}}.] The results follow from Lemma \ref{tracemartin} and Lemma \ref{tracegreen}.
\end{proof}

\begin{proof}[\textbf{Proof of Theorem \ref{subsuperhar}}.] (i) Since $-L_\mu u \geq 0$ in the sense of distributions in $\Omega \setminus K$, there exists a nonnegative Radon measure $\tau$ in $\Omega \setminus K$ such that $-L_\mu u=\tau$ in the sense of distributions. By \cite[Lemma 1.5.3]{MVbook}, $u\in W^{1,p}_{loc}(\xO\setminus K).$
	
	Let $\{O_n\}$ be a smooth exhaustion of $\xO\setminus K$ (see \eqref{Omegan} and \eqref{Kn}). Denote by $G_\mu^{O_n}$ and $P_\mu^{O_n}$ the Green kernel and the Poisson kernel of $-L_\mu$ in $O_n$ respectively (recalling that $P_\mu^{O_n} = -\partial_{\bf n}G_\mu^{O_n}$). Then $u=\BBG_{\mu}^{O_n}[\tau]+ \BBP_\mu^{O_n}[v]$, where $\BBG_\mu^{O_n}$ and $\BBP_\mu^{O_n}$ are the Green operator and the Poisson operator in $\{ O_n\}$ respectively.
	
	Since $\tau$ and $\BBP_\mu^{O_n}[v]$ are nonnegative and $G_{\mu}^{O_n}(x,y)\nearrow G_{\mu}(x,y)$ for any $x\neq y$ and $x,y\in \xO\setminus K$, we obtain $0 \leq \mathbb{G}_{\mu}[\tau]\leq u$  a.e. in  $\Omega \setminus K$. In particular, $0 \leq \mathbb{G}_{\mu}[\tau](x_0) \leq u(x_0)$ where $x_0 \in \Omega \setminus K$ is a fixed reference point. This, together with the estimate $G_\mu(x_0,\cdot) \gtrsim \ei$ a.e. in $\Omega \setminus K$, implies $\tau \in \GTM(\Omega \setminus K;\ei)$.
	
	Moreover, we see that $u-\mathbb{G}_{\mu}[\tau]$ is a nonnegative $L_\mu$-harmonic function in $\xO\setminus K$. Thus by Theorem \ref{th:Rep}, there exists a unique $\nu \in \GTM^+(\partial\xO\cup K)$ such that \eqref{reprweaksol} holds. \smallskip
	
	(ii) Since $-L_\mu u \leq 0$ in the sense of distributions in $\Omega \setminus K$, there exists a nonnegative Radon measure $\tau$ in $\Omega \setminus K$ such that $-L_\mu u =-\tau$ in the sense of distributions. By \cite[Lemma 1.5.3]{MVbook}, $u \in W^{1,p}_{loc}(\xO\setminus K)$. Let $O_n$ and $\BBP_\mu^{O_n}$ be as in (i).
	Then  $u+\BBG_{\mu}^{O_n}[\tau]=\BBP_\mu^{O_n}[v]$. This, together with the fact that $u \geq 0$ and  $\BBP_\mu[u] \leq w$, implies $\BBG_\mu^{O_n}[\tau] \leq w$. By using a similar argument as in (i), we deduce that $\tau \in \GTM(\Omega \setminus K;\ei)$ and there exists $\nu \in \GTM^+(\partial \Omega \cup K)$ such that \eqref{vGK2} holds.
\end{proof}

\subsection{Boundary value problem for linear equations}

We recall that ${\bf X}_\mu(\Omega \setminus K)$ is defined in \eqref{Xmu}. The following result provides an estimate for functions in ${\bf X}_\mu(\Omega \setminus K)$.

\begin{lemma} \label{testfunc-prop}
	For any $\zeta \in {\bf X}_\mu(\Omega \setminus K)$, there holds $|\zeta| \lesssim \ei$ in $\Omega \setminus K$.
\end{lemma}
\begin{proof}
	Let $\zeta \in {\bf X}_\mu(\Omega \setminus K)$ then there exists $f \in L^\infty(\Omega)$ such that $-L_\mu \zeta = f\phi_\mu$ in $\Omega \setminus K$. By using Lemma \ref{existence2} with $b=\am$ and \eqref{eigenfunctionestimates}, we derive $|\zeta| \lesssim \ei$ in $\Omega \setminus K$ (in fact we choose $\gamma=\am$ in \eqref{veryweakest}).
\end{proof}

\begin{lemma}\label{existence3}
	Let $\tau\in\GTM(\Omega \setminus K;\ei).$ Then there exists a unique weak solution $u$ of \eqref{NHLP} with $\xn=0$. Furthermore $u=\mathbb{G}_{\mu}[\tau]$ and there holds
	\begin{align}  \label{esti1}
	\| u \|_{L^1(\Omega;\ei)}  \leq \frac{1}{\lambda_\mu} \| \tau \|_{\GTM(\Omega \setminus K;\ei)}.
	\end{align}
\end{lemma}
\begin{proof}
	\emph{A priori estimate.} Assume $u \in L^1(\Omega;\ei)$ is a weak solution of \eqref{NHLP}. Let $\zeta \in {\bf X}_\mu(\Gw\setminus K)$ be such that $-L_\mu\zeta=\sign( u)\ei$. By Kato's inequality,
	$$-L_\mu |\zeta|\leq -\sign(\zeta)L_\mu \zeta \leq \ei=-L_\mu\left(\frac{1}{\lambda_\mu}\phi_\mu \right).$$
	Hence by Lemma \ref{comparison} we can easily deduce that $|\zeta|\leq \frac{1}{\xl_\xm} \ei$ in $\Omega \setminus K$. This, combined with \eqref{lweakform}, implies \eqref{esti1}. (Here we note that $\nu=0$ in \eqref{lweakform}.) \medskip
	
	\noindent \emph{Uniqueness.} The uniqueness follows directly from \eqref{esti1}. \medskip
	
	\noindent \emph{Existence.} Assume $\tau=fdx$ with $f \in L^\infty(\Omega)$. The existence follows by Lemma \ref{existence1}.
	
	Since $f\in L^\infty(\Omega)$, using a Moser iteration argument similar to the one in \cite[Theorem 2.12]{FMoT2} (see also \cite[Theorem 3.7]{FMoT}) we can show that there exists a positive constant $C>0$ such that $\sup_{x\in\xO\setminus K}|v|\leq C$, which implies $|u(x)|\leq C\ei(x)$ for all $x\in\xO\setminus K$.
	\medskip
	
	\noindent \textit{Next we will show that $u=\BBG_\mu[f]$.} Set $w=\BBG_\mu[f]$ then we can easily show that $w$ satisfies $-L_\xm w=f$ in the sense of distributions in $\Omega \setminus K$ and by Lemma \ref{existence1b}, there exists a positive constant $C$ such that $|w(x)| \leq C\ei(x)$  for all $x\in\xO\setminus K$. Therefore,
	$$ \lim_{\dist(x,F) \to 0}\frac{|u(x)-w(x)|}{\tilde W(x)} \leq  C\lim_{\dist(x,F) \to 0}\frac{\ei(x)}{\tilde W(x)} =0
	$$
	for all compact set $F \subset \partial \Omega \cup K$.
	Furthermore, we note that $|u-v|$ is $L_\mu$-subharmonic in $\Omega \setminus K$. Hence from Lemma \ref{comparison}, we deduce that $|u-w|=0$, i.e. $u=w$ in $\Omega \setminus K$.
	
	Now assume that $\tau =fdx$ with $f\in L^1(\Omega;\ei)$. Let $f_n\in L^\infty(\Omega)$ such that $f_n\rightarrow f$ in $ L^1(\Omega;\ei)$. Put $u_n:=\BBG_\mu[f_n]$ then
	\begin{equation} \label{weakfnnu=0}
	- \int_{\Omega}u_n L_{\xm }\zeta \, dx=\int_{\Omega} f_n \zeta  \, dx
	\qquad\forall \xi \in\mathbf{X}_\xm(\xO\setminus K).
	\end{equation}
	By \eqref{esti1} we can easily prove that $u_n=\BBG_\mu[f_n] \to \BBG_\mu[f]:=u$ in $L^1(\Omega;\ei)$. Then by letting $n \to \infty$ and using Lemma \ref{testfunc-prop}, we deduce the desired result when $f\in L^1(\Omega;\ei)$.
	
	Assume $\tau \in \GTM(\Omega \setminus K;\ei)$. Let $\{f_n\}$ be a sequence in  $L^1(\Omega;\ei)$ such that
	$f_n\rightharpoonup \tau $ in $C_\ei(\Omega \setminus K)$, where $C_\ei(\Omega \setminus K)$ denotes the space of functions $\zeta \in C(\Omega \setminus K)$ such that
	$\ei \zeta  \in L^\infty(\Omega)$. Then proceeding as above we can prove that
	$u_n=\BBG_\mu[f_n] \to \BBG_\mu[\tau]:=u$ in $L^1(\Omega;\ei)$
	and $u$ satisfies \eqref{lweakform} with $\xn=0.$
\end{proof}

\begin{theorem}
	Let $\tau,\rho\in\mathfrak{M}(\xO\setminus K;\ei)$, $\xn \in \mathfrak{M}(\partial\xO\cup K)$ and $f\in L^1(\Omega;\ei)$.  Then there exists a unique weak solution $u\in L^1(\Omega;\ei)$ of \eqref{NHLP}. Furthermore
	\be \label{reprweaksol1}
	u=\mathbb{G}_{\mu}[\tau]+\mathbb{K}_{\xm}[\xn].
	\ee
	There exists a positive constant $C=C(N,\Omega,K,\mu)$ such that
	\begin{align} \label{esti2}
	\|u\|_{L^1(\Omega;\ei)} \leq \frac{1}{\lambda_\mu}\| \tau \|_{\GTM(\Omega \setminus K;\ei)} + C \| \nu \|_{\GTM(\partial\Omega \cup K)} \quad \zeta \in \mathbf{X}_\xm(\xO\setminus K)
	\end{align}
	In addition, if $d\tau=fdx+d\rho$ then, for any $\zeta \in \mathbf{X}_\xm(\xO\setminus K)$ and $\zeta \geq 0$, there hold
	\be\label{poi4}
	-\int_{\Omega}|u|L_{\mu }\zeta \, dx\leq \int_{\Omega}\sign(u)f\zeta\, dx +\int_{\Omega \setminus K}\zeta d|\rho|-
	\int_{\Omega}\mathbb{K}_{\xm}[|\xn|] L_{\xm }\zeta \, dx,
	\ee
	\be\label{poi5}
	-\int_{\Omega}u_+L_{\mu }\zeta \, dx\leq \int_{\Omega} \sign_+(u)f\zeta\, dx +\int_{\Omega \setminus K}\zeta\,d\rho_+-
	\int_{\Omega}\mathbb{K}_{\xm}[\nu_+]L_{\xm }\zeta \,dx,
	\ee
	\be\label{poi6}
	-\int_{\Omega}u_- L_{\mu }\zeta \, dx\leq \int_{\Omega} \sign_-(u)f\zeta\, dx +\int_{\Omega \setminus K}\zeta\,d\rho_- -
	\int_{\Omega}\mathbb{K}_{\xm}[\nu_-]L_{\mu }\zeta \,dx.
	\ee
\end{theorem}
\begin{proof}
	\textit{Existence.}
	The existence and \eqref{reprweaksol1} follow from Lemma \ref{existence3} and Theorem \ref{th:Rep}. \medskip
	
	\noindent \textit{A priori estimate \eqref{esti2}.} By Corollary \ref{cor:pcritical}, $\| \BBK_\mu[|\nu|] \|_{L^1(\Omega;\ei)}  \lesssim \| \nu \|_{\GTM(\partial\Omega \cup K)}$. This, together with \eqref{esti1} and \eqref{reprweaksol1}, implies \eqref{esti2}. \medskip
	
	\noindent \textit{Uniqueness.}  The uniqueness follows from \eqref{esti2}.  \medskip
	
	\noindent \emph{Proof of estimates \eqref{poi4}--\eqref{poi6}.}
	Assume $d\tau=fdx+d\rho$ and let $\{O_n\}$ be a smooth exhaustion of $\xO\setminus K$ and . Let $v_\tau^n$ be the solution of
	$$ \left\{ \BAL
	-L_{\xm }^{O_n}v&=0\qquad&&\text{in } O_n\\
	v&=\mathbb{G}_{\mu}[\tau]\qquad&&\text{on } \prt O_n,
	\EAL \right.
	$$
	and $w_\xn=\mathbb{K}_{\xm}[\nu]$. Then by the uniqueness, $u=\mathbb{G}^{O_n}_{\mu}[\tau|_{O_n}]+v_\tau+w_\xn$ and $|u|\leq \mathbb{G}_{\mu}[|\tau|]+w_{|\xn|}$ $\mathcal{H}^{N-1}$-a.e. on $\partial O_n$ (here on $\mathcal{H}^{N-1}$ denotes the Hausdorff measure on $\partial O_n$).
	
	For any nonnegative $\eta \in C_0^2(O_n),$ by \cite[Proposition 1.5.9]{MVbook},
	\begin{align}
	-\int_{O_n}|u|L_\xm \eta \leq \int_{O_n}\sign(u)f\eta dx+ \int_{O_n}\sign(u)\eta d|\rho|-\int_{\partial O_n}|u|\frac{\partial\eta}{\partial {\bf n}^n}dS
	\end{align}
	where ${\bf n}^n$ is the unit outer normal vector on $\partial O_n.$
	
	Since $|u| \leq \BBG_\mu[|\tau|] + w_{|\nu|}$ a.e. on $\partial O_n$ and $\frac{\partial \eta}{\partial {\bf n}^n} \leq 0$ on $\partial O_n$, by using integration by parts, we obtain
	\begin{align*}
	-\int_{\partial O_n}|u|\frac{\partial\eta}{\partial {\bf n}^n}dS \leq -\int_{\partial O_n}(\mathbb{G}_{\mu}[|\tau|]+w_{|\xn|})\frac{\partial\eta}{\partial {\bf n}^n}dS=-\int_{O_n}(v_{|\tau|}^n+w_{|\xn|})L_\xm\eta dx.
	\end{align*}
	Hence
	\begin{align}
	-\int_{O_n}|u|L_\xm \eta dx \leq \int_{O_n}\sign(u)f\eta dx+ \int_{O_n}\sign(u)\eta d|\rho|-\int_{O_n}(v_{|\tau|}^n+w_{|\xn|})L_\xm\eta dx. \label{ee1}
	\end{align}
	
	Let $\xz\in \mathbf{X}_\xm(\xO\setminus K)$, $\zeta > 0$ in $\xO\setminus K$. Let $z_n$ and $\zeta_n$ be respectively solutions of
	$$ \left\{ \BAL
	-L_{\xm }z_n&=-L_{\xm }\xz\qquad&&\text{in } O_n\\
	z_n&=0\qquad&&\text{on } \prt O_n,
	\EAL \right. \qquad
	 \left\{ \BAL
	-L_{\xm }\zeta_n&=-\sign(z_n) L_{\xm }\xz\qquad&&\text{in } O_n\\
	\zeta_n&=0\qquad&&\text{on } \prt O_n.
	\EAL \right.$$
	
	By Kato's inequality, $-L_\mu |z_n|\leq -\sign(z_n) L_\mu z_n$ in the sense of distributions in $O_n$. Hence by a comparison argument, we have that $|z_n|\leq \xz_n$ in $O_n$. Furthermore it can be checked  that $z_n \to \zeta$ and $\xz_n \to \xz$ in $L^1(\Omega;\ei)$ and locally uniformly in $\xO\setminus K$.
	
	Now note that \eqref{ee1} is valid for any nonnegative solution $\eta\in C_0^2(O_n)$. Thus we can use $\zeta_n$ as a test function in \eqref{ee1} to obtain
	\begin{equation} \label{ee2} \begin{aligned}
	-\int_{O_n}|u|\sign(z_n) L_\xm \xz dx &\leq \int_{O_n}\sign(u)f\xz_n dx+ \int_{O_n}\sign(u)\xz_n d|\rho| \\
	&-\int_{O_n}(v_{|\tau|}^n+w_{|\xn|})\sign(z_n)L_\xm\xz dx.
	\end{aligned} \end{equation}
	
	Also, since $\mathbb{G}_{\mu}[|\tau|]=\mathbb{G}_{\mu}^{O_n}[|\tau||_{O_n}]+v_{|\tau|}^n$ a.e. in $O_n$, we deduce that $v_{|\tau|}^n\rightarrow 0$ in $L^1(\Omega;\ei)$ as $n \to \infty$. Letting $n \to \infty$ in \eqref{ee2}, we obtain \eqref{poi4} since $\zeta>0$ in $\Omega \setminus K$. Estimates \eqref{poi5} and \eqref{poi6} follow by adding \eqref{poi4} and \eqref{lweakform}. Thus the proof is complete when $\xz$ is positive.
	
	If $\xz$ is nonnegative we set $\xz_\xe=\xz+\xe\ei.$ Then estimates \eqref{poi4}--\eqref{poi6} are valid for $\xz_\xe$ for any $\xe>0.$  The desired result follows by sending $\varepsilon \to \infty$.
\end{proof}	


\end{document}